\documentclass[fleqn,12pt,twoside]{article}
\usepackage{a4wide}
\usepackage[active]{srcltx}
\usepackage{graphicx,palatino,pifont,times}
\usepackage{oldgerm}
\usepackage{amsmath,amsthm}
\usepackage{amssymb}
\usepackage{amstext}
\usepackage{epsfig}
\usepackage{multirow}
\usepackage{array}
\usepackage{enumerate}
\usepackage{cite}
\usepackage[active]{srcltx}
\usepackage{lscape}

\newtheorem{thm}{Theorem}[section]

\newtheorem{rem}{Remark}[section]
\newtheorem{prop}{Proposition}[section]
\newtheorem{cor}{Corollary}[section]

\catcode`\@=11
\def\theequation{\@arabic{\c@section}.\@arabic{\c@equation}}
\catcode`\@=12
\begin{document}
\begin{center}
{\Large \bf { Shape Preserving Rational Cubic Spline Fractal
Interpolation }}
\end{center}
\begin{center}
{\sc{A. K. B. Chand and P. Viswanathan }}\\
Department of Mathematics\\
Indian Institute of Technology Madras\\
Chennai - 600036, India \\
Email:  chand@iitm.ac.in ;  amritaviswa@gmail.com\\
Phone: 91-44-22574629\\
Fax : 91-44-22574602\\
\end{center}
\begin{abstract}
 Fractal Interpolation Functions (FIFs) developed through
Iterated Function Systems (IFSs) offer more versatility than the
classical interpolants. However, the application of FIFs in the
domain of shape preserving interpolation is not fully addressed so
far. Among various interpolation techniques that are available in
the classical numerical analysis, the rational interpolation
schemes are well suited for the shape preservation and shape
modification problems. Consequently, we introduce a new class of
rational cubic spline FIFs that involve tension parameters as a
common platform for the shape preserving interpolation and the
fractal interpolation to work together. Suitable conditions on the
parameters are developed so that the rational fractal interpolant
retain the monotonicity and convexity properties inherent in the
given data. With some suitable hypotheses on the original data
generating function, the convergence analysis of the rational
cubic spline FIF is carried out. Due to the presence of  scaling
factors in the rational cubic spline fractal interpolant, our
approach generalizes the classical results on the shape preserving
rational interpolation by Delbourgo and Gregory [SIAM J. Sci.
Stat. Comput., 6 (1985), pp. 967-976]. Furthermore, for preserving
shape of a data set wherein the variables representing the
derivatives have varying irregularity, the present schemes
outperform their classical counterparts. Several examples are
supplied to support the practical value of the method.\\
{\bf This paper was submitted to SIAM Journal of Numerical Analysis on April 14, 2013 after initial submission to Constructive Approximation on July 18, 2012.}
\end{abstract}
{\bf KEYWORDS :} Iterated Function System, Fractal Interpolation
Function, Rational Cubic Spline FIF, Convergence, Shape
Preservation, Monotonicity, Convexity.
\\\\
{\bf AMS Subject Classifications:} 28A80, 26C15, 26A48, 26A51,
65D05, 65D17.
\section{Introduction}\label{r1ps1}
In the classical Numerical Analysis, there are several
interpolation methods that can be applied to a specific data set,
according to the assumptions that underlie the model we
investigate. However, if the given data set is more complex and
irregular (for instance, data sampled from real-world signals such
as financial series, seismic data, and bioelectric recordings),
then the traditional interpolants may not provide satisfactory
results. To address this issue, Barnsley \cite{B} introduced a
class of functions called FIFs using the notion of IFS. FIFs aim
mainly at data which present details at different scales or some
degree of self-similarity. These characteristics imply an
irregular structure to the interpolant. The main differences of
FIFs from the traditional interpolants reside in: a) the
definition in terms of functional equation which implies a self
similarity in small scales; b) the iterative construction instead
of using an analytical formula and c) the usage of some
parameters, which are usually called scaling factors that are
strongly related to the fractal dimension of the interpolant.
Later, Barnsley and Harrington\cite{BH} observed that if the
problem is of differential type, then the parameters of the IFS
may be chosen suitably so that the corresponding FIF is smooth.
This observation initiated a striking relationship between the
classical splines and the fractal functions. Smooth FIFs (fractal
splines) constitute an advance in the techniques of approximation,
since the classical methods of real-data interpolation can be
generalized by means of smooth fractal techniques (see, for
instance, \cite{CK, N, NS, CN1}). Further, if the experimental
data are approximated by a $\mathcal{C}^r$-FIF $S$, then the
fractal dimension of $S^{(r)}$ provides a quantitative parameter
for the analysis of the data, allowing to compare and discriminate
experimental processes. Though FIFs are primarily applied for a
self-affine/self-similar data set, their extensions namely hidden
variable FIFs \cite{BM} and coalescence hidden variable FIFs
\cite{CK2} can be used to simulate curves that are non-self-affine
or partly self-affine and partly non-self-affine in nature. Due to
these versatility and flexibility, the theory of FIFs has evolved
beyond its mathematical framework and has become a powerful and
useful tool in the applied sciences as well as engineering.

\par
The central focus of interpolation (traditional or fractal) is to
construct a continuous function  that fits the given points
obtained by sampling or experimentation. However, to obtain a
valid physical interpretation of the underlying process, it is
important to develop interpolation schemes that inherit certain
properties from the prescribed data set. Examples of few such
prevalent features are positivity, monotonicity, and convexity.
Constraining the range of an interpolation function so as to yield
a credible visualization of the data by adhering to these
intrinsic characteristics is generally referred to as  \emph{shape
preserving interpolation}. There are multitudes of  classical
interpolation methods that honour shape properties inherent in the
data.  In what follows, we shall provide some pioneering works in
this field.
\par
Research on shape preserving interpolation has been originated
with some existence-type results by Wolibner\cite{W} and
Kammerer\cite{K}. These results do not provide any additional
information on the shape preserving polynomial. A constructive
approach to the shape preserving interpolation using hyperbolic
tension splines was popularized by Schweikert \cite{SCH}. Main
issues connected with the hyperbolic tension splines are: (i)
development of an automatic algorithm for the choice of free
parameters is complicated; (ii) it is computationally complex to
work with, especially for very large/small values of tension
parameters involved in it. Polynomial splines gain shape
properties either (i) by addition of extra knots (see \cite{SC,
FT, SH,MS}), which may not be effective in terms of computational
economy, or (ii) by perturbing derivative parameters (see, for
instance, \cite{FC, FB}), which make the method unsuitable for
Hermite data, where the given derivative values are also to be
interpolated. In shape preservation and shape control, rational
splines provide an acceptable alternative to the
polynomial/hyperbolic splines. Wide applicability of the rational
interpolants may be attributed to their ability to receive free
parameters (which may be used for shape control) in their
structure, ability to accommodate a wider range of shapes than the
polynomial family, excellent asymptotic properties, capability to
model complicated structures, better interpolation properties, and
excellent extrapolating powers. Gregory and Delbourgo popularized
the shape preserving rational interpolation methods through a
series of papers \cite{DG1, DG2,DG4, D2,G}. These works stimulated
a large amount of research in the direction of shape preserving
rational spline interpolation. For brevity, the reader is referred
to \cite {JCC, XH, S4,GS}.
\par
These non-recursive shape preserving interpolation techniques,
 in general, produce
smooth interpolants whose derivatives are also smooth except
possibly at some finite number of points. However, in practice,
there are many situations where the variable in the data possesses
certain shape properties, and at the same time, variables
representing the derivatives may be irregular. For instance, a
sphere falling in a wormlike micellar solution does not approach a
steady terminal velocity, instead, it undergoes continual
oscillations as it falls \cite{JA}. Hence, to simulate the
displacement and velocity profiles of such motions,
monotonicity/positive interpolants with varying irregularity
(fractality) in the derivatives may be advantageous. Similarly, in
nonlinear control systems \cite{WS} (say, for instance, the motion
of a pendulum on a cart) monotonicity/convexity preserving
interpolants with varying irregularity in the second derivative
may be desirable for the study of acceleration. Therefore, it is
useful to develop smooth FIFs (which are known to have fractality
in their derivatives) that retain the intrinsic shape properties
of the data set. From the knowledge gained from the classical
shape preserving polynomial interpolation techniques, it is felt
that preserving fundamental shape properties via polynomial FIFs
would be difficult or impossible. Thus, for an initial exposition
of FIFs to the field of shape preserving interpolation, rational
FIFs seem to be an appropriate medium.
\par
With these motivations, the capability of FIFs to generalize
smooth classical interpolants, and the effectiveness of rational
function models in shape preservation are intertwined to provide a
new solution to the shape preserving interpolation problem from a
fractal perspective. We construct a $\mathcal{C}^1$-rational cubic
spline FIF with one family of shape parameters in section
\ref{r1ps31}. To demonstrate the effectiveness of the rational
cubic spline fractal interpolation scheme, the convergence results
are discussed in section \ref{r1ps32}. The rationale behind
selecting a rational FIF with shape parameters instead of widely
studied polynomial FIFs is the following. If the scaling factors
tend to zero and the shape parameters tend to infinity, then the
rational FIF converges to the piecewise linear interpolant for the
data. This tension effect ensures that the FIF can be used for
shape preserving interpolation. Section \ref{r1ps4} provides an
automatic selection of the parameters that culminate in
interactive algorithms to preserve monotonicity/convexity of the
data. These algorithms take full advantage of the flexibility that
the fractal splines permit. By suitable choice of the shape
parameters that verify the monotonicity condition, our cubic
spline FIF reduces to a lower degree rational spline FIF that
generalizes the classical rational spline interpolant studied in
\cite{DG1}. With  special choices of the shape parameters
satisfying convexity condition, our cubic/quadratic rational FIF
reduces to lower order form, which provides the fractal
generalization of the classical rational interpolant discussed in
\cite{D2}. Again, by proper choices of the scaling factors and the
shape parameters, our rational cubic FIF degenerates to the
classical rational cubic interpolating function introduced in
\cite{DG4}. Therefore, the present paper offers a novel idea of
setting a common platform for the fractal interpolation and the
shape preserving interpolation to operate together, and in the
process collectively generalizes  three different classical
rational interpolation schemes available in the literature. In
section \ref{r1ps5}, some remarks and possible extensions are
made. The effectiveness of our shape preserving fractal
interpolation schemes is illustrated with suitably chosen
numerical examples and graphs in section \ref{r1ps6}.
\section{Basics of Polynomial Spline FIF}\label{r1ps2}
\setcounter{equation}{0}  To equip ourselves with the requisite
general material for the construction of the desired rational
spline FIF, we shall reintroduce the polynomial spline FIF. A
complete and rigorous treatment can be found in \cite{B,B1,BH}.
\subsection{Fractal Interpolation Functions}\label{r1ps21}
Let $\{(x_i,y_i) \in I \times \mathbb{R} : i= 1,2,\dots,N \}$ be a
real data set, where $ x_1 < x_2 < \dots <x_N$ is a partition of
$I = [x_1, x_N]$. Set, $K= I \times D$, where $D$ is a large
enough compactum in $\mathbb{R}.$ Let $ J = \{1,2,\dots,N-1 \}$,
and $L_i : I \longrightarrow I_i = [x_i, x_{i+1}]$ be  affine maps
satisfying:
\begin{equation}\label{r1}
L_i(x_1) = x_i, \quad L_i(x_N) = x_{i+1}, \; i \in J,
\end{equation}
and  $ F_i : K \longrightarrow D $ be  continuous functions such
that:
\begin{equation}\label{r2}\left.
\begin{split}
F_i(x_1,y_1)   = y_i,\;
F_i(x_N,y_N)   =  y_{i+1} \\
|F_i(x,y)-F_i(x,y^*)| \le  |\alpha_i| |y-y^*|
\end{split}\right\}, \; i \in J,
\end{equation}
where  $(x,y),  (x,y^*) \in K $, $0\le\ |\alpha_i| \le \kappa <1$
for all $i \in J $, and $\kappa$ is a fixed real constant. Define
$ w_i(x,y) = \big(L_i(x), F_i(x,y)\big)$ for $ i \in J.$ It is
known \cite{B} that there exists a metric on $\mathbb{R}^2$,
equivalent to the Euclidean metric, with respect to which $w_i$, $
i \in J$, are contractions. The collection $\{K; w_i, i \in J\}$
is termed as an Iterated Function System (IFS). On
$\mathcal{H}(K)$, the set of all nonempty compact subsets of $K$,
endowed with the Hausdorff metric, define a set valued map
$W(A)=\underset{i\in J}\bigcup w_i (A)$. Then, $W$ is a
contraction map on the complete metric space $\mathcal{H}(K)$.
Thanks to Banach Fixed Point Theorem, there exists a unique set
$G\in \mathcal{H}(K)$ such that $W(G)=G$. The set $G$ is termed as
the attractor or deterministic fractal corresponding to the IFS
$\{K; w_i, i\in J\}$. The definition of a FIF originates from the
following proposition:
\begin{prop}(Barnsley \cite{B}) \label{P1}
 The IFS $\{ K ; w_i , i \in J \}$ has a unique attractor $G$ such that $G$ is the graph
of a continuous function $ f: I \rightarrow \mathbb{R}$ which
interpolates the data $\{(x_i,y_i): i= 1,2,\dots,N \}$, i.e.,
$G=\{(x, f(x)): x \in I\}$ and $f(x_i)= y_i~ \text{for}~
i=1,2,\dots,N$.
\end{prop}\label{prop1}
\noindent The function $f$ in Proposition \ref{P1} is called a
\emph{fractal interpolation function} corresponding to the IFS $\{
K ; w_i , i \in J \}$. The adjective \emph{fractal} is used to
emphasize that $G=\text{graph}(f)$ may have noninteger
Hausdorff-Besicovitch dimension. But $f$ may be many times
differentiable (see section \ref{r1ps22}). Since $G$ is a union of
transformed copies of itself, i.e., $G=\underset{i \in J}\bigcup
w_i(G)$, an alternative name for a fractal function could be a
\emph{self-referential function}. The characterization of a graph
of a FIF by an IFS leads to a recursive construction of $f$ using
the following functional equation \cite{B}:
\begin{equation}\label{r5}
f(x)=F_i\big(L_i^{-1}(x), f \circ L_i^{-1}(x)\big), \ \ x \in I_i,
\ \ i \in J.
\end{equation}
\noindent The following special class of  IFS has received wide
attention in the literature:
\begin{equation}\label{r6} \left.
 \begin{split}L_i(x)&=a_i x+b_i\\
F_i(x,y)&=\alpha_i y +R_i(x) \end{split}\right \}, i \in J,
\end{equation}
\noindent where $\alpha_i$, $i \in J,$ are parameters satisfying $
|\alpha_i| \le \kappa <1$ and $R_i : I \to \mathbb{R}, \; i \in
J$, are suitable polynomials satisfying (\ref{r2}). The multiplier
$\alpha_i$ is called a scaling factor of the transformation $w_i$
and $\alpha=(\alpha_1,\alpha_2,\dots, \alpha_{N-1})$ is the scale
vector of the IFS. As mentioned earlier in the introductory
section, a FIF is defined in a constructive way through iterations
instead of descriptive one, usually a formula, provided by the
classical methods. Consequently, evaluation of a FIF at a given
point needs, in general,  an iteration process. An explicit
representation (in terms of an infinite series) of a FIF $f$
corresponding to a general IFS (\ref{r6}) defined on $I=[0,1]$ is
given in \cite{SCh}. For a representation of $f$ as the uniform
limit of a sequence of operators, the reader is referred to
\cite{CK3}.
\subsection{Polynomial Spline FIFs}\label{r1ps22}
 For a prescribed  data set, a  polynomial FIF with
$\mathcal{C}^r$-continuity is obtained as the fixed point of IFS
(\ref{r6}), where the scaling factors $\alpha_i$ and the
polynomials $R_i$ involved in (\ref{r6}) are chosen according to
the following proposition.
\begin{prop} \label{P2}(Barnsley et al. \cite{BH})
Let $x_1<x_2<\dots <x_{N}$ and $L_i(x) = a_i x + b_i$, $i\in J$,
be the affine functions satisfying (\ref{r1}). Let
$F_i(x,y)=\alpha_i y + R_i(x), i\in J$, satisfy (\ref{r2}).
Suppose that for some integer $p\geq0$, $|\alpha_i|< a_i^p$ and
$R_i\in \mathcal{C}^p(I)$, $i\in J$. Let
\begin{eqnarray*}
F_{i,k}(x,y)=\frac{\alpha_i y+ R_i^{(k)}(x)}{a_i^k}, \;\; y_{1,k}
= \frac
{R_1^{(k)}(x_1)}{a_1^k-\alpha_1},\:y_{N,k}=\frac{R_{N-1}^{(k)}(x_N)}{a_{N-1}^k-\alpha_{N-1}};
~~~k = 1,2,\dots,p.
\end{eqnarray*}
If $ F_{i-1,k}(x_N, y_{N,k})=F_{i,k}(x_1,y_{1,k})$  for $\ i =2,3,
\dots,N-1$ and  $\ k =1,2,\dots,p$, then the IFS
$\big\{\mathbb{R}^2;\big(L_i(x),F_i(x,y)\big),i\in J\big\}$
determines a FIF ${f}\in \mathcal{C}^p[x_1,x_N]$, and $f^{(k)}$ is
the FIF determined by
$\big\{\mathbb{R}^2;\big(L_i(x),F_{i,k}(x,y)\big),i\in J\big\}$
for $ k = 1,2,\dots,p$.
\end{prop}
\section{Rational Cubic Spline FIFs Involving One Family of Shape Parameters}\label{r1ps3}
\setcounter{equation}{0} In section \ref{r1ps31}, we construct
$\mathcal{C}^1$-rational spline FIFs where the inhomogeneous terms
are  rational functions with cubic numerators and preassigned
quadratic denominators. In section \ref{r1ps32}, error analysis of
the rational cubic spline FIF is given with the assumption that
the data defining function $f \in \mathcal{C}^4$. Further, by
admitting a relatively weaker condition on the data generating
function $f$, namely $f\in$ $\mathcal{C}^1$, the uniform
convergence of the classical rational cubic spline is established.
This serves as an addendum to the convergence results by Delbourgo
and Gregory \cite{DG4}, and it is utilized to establish the
uniform convergence of the developed rational cubic spline FIF.
\subsection{Existence of $\mathcal{C}^1$-rational Cubic Spline FIF}\label{r1ps31}
\begin{thm}\label{theorem1}
Suppose a data set $\{(x_i,y_i,d_i):i=1,2,\dots, N\}$ is given,
where $x_1 < x_2 < \dots <x_N$. Consider the rational  IFS
$\{(L_i(x), F_i(x,y): i \in J\}$ where $L_i(x)=a_ix+b_i$ and
$F_i(x,y)=\alpha_i y + R_i(x)$, $|\alpha_i| \le \kappa a_i$,
$0\le\kappa< 1$, $i \in J$. Further, let
$R_i(x)=\frac{P_i(x)}{Q_i(x)}$ where $P_i(x)$ is a cubic
polynomial and $Q_i(x)\neq 0$ is a preassigned quadratic
polynomial such that $F_i(x_1,y_1)=y_i$, $F_i(x_N,y_N)=y_{i+1}$
are satisfied. With $F_{i,1}(x,y)=\frac{\alpha_i y+
R_i^{(1)}(x)}{a_i}$, let $F_{i,1}(x_1,d_1)=d_i$ and
$F_i(x_N,d_N)=d_{i+1}$. Then a $\mathcal{C}^1$-rational cubic
spline FIF $S$ satisfying $S(x_i)=y_i$, $S^{(1)}(x_i)=d_i$,
$i=1,2,\dots, N$ exists, and it is unique for a fixed choice of
the shape parameters and the scaling factors.
\end{thm}
\begin{proof}
 Set $I=[x_1,x_N]$, and $h_i=x_{i+1}-x_i$, $i \in J$.
Consider the IFS $\mathcal{I}=\big\{\big(L_i(x), F_i(x,y)\big): i
\in J\big\}$ where $L_i(x)= a_i x+ b_i$ satisfy (\ref{r1}), and
$F_i(x,y)=\alpha_i y + R_i(x)$ fulfill the \emph{join-up}
conditions $F_i(x_1,y_1)=y_i$, $F_i(x_N,y_N)=y_{i+1}$. Let
\begin{equation*}
 R_i(x)\equiv R_i^*\big(x_1+ \theta(x_N-x_1)\big)= \frac{A_i(1-\theta)^3  + B_i \theta (1- \theta)^2 + C_i\theta^2 (1-\theta) + D_i\theta^3}
{1 + (r_i -3)  \theta (1-\theta)}, \; i \in J.
\end{equation*}
Consider $\mathcal{F} := \{ g\in$ $\mathcal{C} (I) | \; g(x_1) =
y_1 \; \text{and} \; g(x_N) = y_N \}$. The  uniform metric
$d(g,h):=\sup\{|g(x)-h(x)|: x \in I\}$ completes $\mathcal{F}$.
The IFS $\mathcal{I}$ induces a contraction map
$T:\mathcal{F}\rightarrow \mathcal{F}$
$$
(Tg)\big(L_i(x)\big):=F_i\big(x,g(x)\big),\; x\in I.$$  The
contraction map $T$ has a unique fixed point $S \in \mathcal{F}$,
which satisfies the functional equation:
\begin{equation}\label{r9}
\begin{split}
 S\big(L_i(x)\big) &= F_i\big(x, S(x)\big),\\
 &= \alpha_i S(x) + \frac{A_i(1-\theta)^3  + B_i\theta (1- \theta)^2 + C_i \theta^2 (1-\theta) + D_i\theta^3}
{1 + (r_i -3)  \theta (1-\theta)}.
\end{split}
\end{equation}
The conditions $F_i(x_1,y_1)=y_i$, $F_i(x_N,y_N)=y_{i+1}$ can be
reformulated as the interpolation conditions $S(x_i)=y_i$,
$S(x_{i+1})=y_{i+1}$, $i \in J$.  Note that: (i) the affine map
$L_i$ satisfies $L_i(x_1)=x_i$ and $L_i(x_N)=x_{i+1}$, (ii)
$x=x_1$ and $x=x_N$ correspond to $\theta=0$ and $\theta=1$
respectively. Therefore, the interpolatory conditions determine
the coefficients $A_i$ and $D_i$ as follows. Substituting $x =
x_1$ in (\ref{r9}) we get
\begin{equation*}
\begin{split}
 S(L_i(x_1)) & = \alpha_i S(x_1) + A_i,  \\
\implies & y_i  = \alpha_i y_1 + A_i,  \\
\implies & A_i  = y_i - \alpha_i y_1.
\end{split}
\end{equation*}
Similarly, taking $x = x_N$ in (\ref{r9}) we obtain $ D_i  =
y_{i+1}- \alpha_i y_N$. \\Now we make $S \in \mathcal{C}^1$  by
imposing the conditions prescribed in
Barnsley-Harrington theorem (see Proposition \ref{P2}).\\
Assume $|\alpha_i| \le \kappa a_i$, $ i \in J$, where $0 \le
\kappa < 1$. We have $R_i \in \mathcal{C}^1(I)$. Adhering to the
notations of Proposition \ref{P2}, we let:
\begin{equation*}
\begin{split}
F_{i,1}(x, y)&= \frac{\alpha_i y+ R_i^{(1)}(x)}{a_i},\\
y_{1,1}=d_1,\; y_{N,1}&=d_N,\; F_{i,1}(x_1, d_1)=d_i,\;
F_{i,1}(x_N,d_N)=d_{i+1};\; i \in J.
\end{split}
\end{equation*}
Then, by Proposition \ref{P2} the FIF $S$ belongs to the class
$\mathcal{C}^1(I)$. Further, $S^{(1)}$ is the fractal function
determined by the IFS $\mathcal {I}^*\equiv\big\{\big(L_i(x),
F_{i,1}(x,y)\big): i \in J\big\}$. Consider $\mathcal{F}^* := \{g
\in \mathcal{C}(I): g(x_1)=d_1~\text{and}~g(x_N)=d_N\}$ endowed
with the uniform metric. The IFS $\mathcal{I}^*$ induces a
contraction map $T^*: \mathcal{F}^* \rightarrow \mathcal{F}^*$
$$(T^*g^*)\big(L_i(x)\big):=F_{i,1}\big(x, g^*(x)\big),\; x \in I.$$ The fixed
point of $T^*$ is $S^{(1)}$. Consequently, $S^{(1)}$ satisfies the
functional equation:
\begin{equation}\label{r9b}
\begin{split}
 S^{(1)}\big(L_i(x)\big) &= F_{i,1}\big(x, S^{(1)}(x)\big),\\
 &= \frac{\alpha_i S^{(1)}(x) + R_i^{(1)}(x)}{a_i}.
\end{split}
\end{equation}
The  conditions on the map $F_{i,1}$, namely
$F_{i,1}(x_1,d_1)=d_i$ and $F_{i,1}(x_N,d_N)=d_{i+1}$ can be
reformulated as the derivative conditions
$S^{(1)}(x_i)=d_i$ and $S^{(1)}(x_{i+1})=d_{i+1}$, $i \in J$.\\
Applying $x = x_1$ in (\ref{r9b}) we obtain
\begin{equation*}
\begin{split}
 S^{(1)}(L_i(x_1)) a_i &  = \alpha_i S^{(1)}(x_1) + \frac{B_i - r_i A_i}{x_N - x_1}, \\
\implies & d_i a_i (x_N - x_1)   =  \alpha_i d_1  (x_N - x_1) + B_i -   r_i (y_i - \alpha_i y_1),  \\
\implies & B_i = [r_i y_i + h_i d_i] - \alpha_i [r_i y_1 + d_1
(x_N - x_1)].
\end{split}
\end{equation*}
Similarly, the substitution  $x = x_N$ in (\ref{r9b}) yields
\begin{equation*}
C_i  = [r_i y_{i+1} - h_i d_{i+1}] - \alpha_i [r_i y_N - d_N (x_N
- x_1) ].
\end{equation*}
Therefore, the desired $\mathcal{C}^1$-rational cubic spline FIF
is described as:
 \begin{equation}\label{r10}
 S(L_i(x)) = \alpha_i S(x) + \frac{P_i(x)} {Q_i(x)},
\end{equation}
 $P_i(x)\equiv P_i^* (\theta) =   (y_i - \alpha_i y_1)  (1-\theta)^3 + \{ [r_i y_i + h_i d_i] - \alpha_i [r_i y_1 + d_1 (x_N - x_1)  ] \} \theta (1- \theta)^2 $ \\
\hspace*{1.5cm}$+  \{ [r_i y_{i+1} - h_i d_{i+1}] - \alpha_i [r_i y_N - d_N (x_N - x_1)  ] \} \theta^2 (1-\theta) + (y_{i+1} - \alpha_i y_N) \theta^3$,\\
$Q_i(x)\equiv Q_i^*(\theta) = 1 + (r_i -3)  \theta (1-\theta),~ \theta = \frac{x-x_1}{x_N - x_1} $.\\
The parameters $r_i> -1$ can be effectively utilized for the shape
modification and shape preservation of the
$\mathcal{C}^1$-rational cubic spline FIF, and hence referred to
as the shape parameters. \\ Since the FIF $S$ in (\ref{r10}) is
derived as a solution of the fixed point equation $Tg = g$, the
solution is unique for a given choice of the scaling factors and
the shape parameters.
\end{proof}
\begin{rem}\label{rem1}
Using the notations
\begin{equation*}\left.
\begin{split}
E_i &:= \frac{\theta (1-\theta)h_i \{(2\theta-1) \Delta_i +
(1-\theta) d_i - \theta  d_{i+1}\}}{1+(r_i-3)\theta
(1-\theta)},\\
F_i &:=\frac{\theta (1-\theta)\{(2\theta-1)(y_N-y_1) +
(1-\theta)(x_N-x_1) d_1- \theta (x_N-x_1) d_N\}}{1+(r_i-3)\theta
(1-\theta)},
\end{split}\right\}
\end{equation*}
the rational cubic spline FIF in (\ref{r10}) is rewritten as
\begin{equation}\label{r10a}
S(L_i(x)) = \alpha_i S(x)+ \{(1-\theta)y_i+ \theta y_{i+1}+ E_i\}
-\alpha_i \{(1-\theta)y_1+ \theta y_N+ F_i\}
\end{equation}
As the shape parameters $r_i$ are increased, $E_i$ and $F_i$
converges to zero. Thus from (\ref{r10a}), it follows that as the
scaling factors $\alpha_i\rightarrow 0$ and the shape parameters
$r_i\rightarrow +\infty $, the rational cubic spline FIF $S$
converges to the piecewise linear interpolant corresponding to the
given data set. This tension effect ensures that the proposed
rational cubic spline FIF can be used to construct shape
preserving interpolants.
\end{rem}
\begin{rem}\label{rc}
If $\alpha_i = 0 $ for all $ i \in J$, then $S$ reduces to the
piecewise defined  $\mathcal{C}^1$-rational cubic spline $s$
discussed in \cite{DG4}. Therefore, $S$ can be considered as an
extension of the powerful rational cubic spline interpolant. To
illustrate this we proceed as follows. With $\alpha_i=0$ for all
$i \in J$, (\ref{r10}) reduces to
 \begin{equation}\label{r11}
 S(L_i(x)) = \frac{ y_i (1-\theta)^3 + (r_i y_i + h_i d_i)  \theta (1- \theta)^2 + (r_i y_{i+1} - h_i d_{i+1}) \theta^2 (1-\theta) + y_{i+1} \theta^3 }{1 + (r_i -3)  \theta (1-\theta)},  i \in J.
 \end{equation}
 Since $\frac {L_i^{-1}(x) -x_1}{x_N-x_1}= \frac{x-x_i}{h_i}$,
from (\ref{r11}), for $x \in I_i=[x_i, x_{i+1}]$, we have
\begin{equation}\label{r11a}
S(x)\equiv s_i(x) = \frac{ y_i (1-\phi)^3 + (r_i y_i + h_i d_i)
\phi (1- \phi)^2 + (r_i y_{i+1} - h_i d_{i+1}) \phi^2 (1-\phi) +
y_{i+1} \phi^3 }{1 + (r_i -3)  \phi(1-\phi)},
\end{equation}
where  $\phi= \frac{x-x_i}{h_i}$ is a localized variable. The
rational cubic spline $s\in \mathcal{C}^1(I)$ is defined by
$s\big|_{I_i}=s_i$, $i \in J$. With $r_i=3$ and the scaling
factors satisfying $|\alpha_i| \le \kappa a_i$, $\forall~ i \in
J$, our discussion on the rational cubic spline FIF gives a simple
constructive approach to the $\mathcal{C}^1$-cubic Hermite FIF.
Again, when $\alpha_i=0~ \text{and}~ r_i=3, \forall~ i \in J$, the
proposed rational cubic spline FIF recovers the classical
piecewise cubic Hermite interpolant.
\end{rem}
\subsection{Convergence  Analysis of Rational Cubic Spline FIFs}\label{r1ps32}
 With mild conditions on the scaling factors, we establish that
 the rational cubic spline FIF $S$ possesses the same convergence
 properties as that of its classical counterpart $s$. Since
$S$ does not possess a closed form expression, standard methods
such as  Taylor series analysis, Cauchy remainder form, and Peano
Kernel theorem cannot be employed to establish its convergence.
Instead, we derive the convergence of $S$ to the original function
$f$ using the convergence results for its classical counterpart
$s$ and the uniform distance between $S$ and $s$ via the triangle
inequality:
\begin{equation}\label{r14a}
 \| f-S \|_{\infty} \le \|f - s
\|_{\infty} + \|s - S \|_{\infty}.
\end{equation}
\begin{thm}\label{t1}
Let $S$ and $s$,  respectively, be the rational cubic spline FIF
and the classical rational cubic spline for the original function
$ f \in$ $\mathcal{C}^4(I)$ with respect to the interpolation data
$\{(x_i, y_i) : i = 1,2, \dots, N \}$, and let
$S^{(1)}(x_i)=s^{(1)}(x_i)=d_i$. Suppose that the rational
function $R_i$  involved in the IFS generating the FIF $S$
satisfies $ \big| \frac{\partial R_i(\tau_i, \phi)}{\partial
\alpha_i} \big| \le Z_0 $ for $|\tau_i| \in (0, \kappa a_i)$, all
$i \in J$, and for some real constant $Z_0$. Then,
\begin{equation}\label{r21}
\begin{split}
\|f - S \|_{\infty} \le & \frac{h}{4c}~\underset{i \in  J} \max
\big \{ |y_i^{(1)} - d_i|,  |y_{i+1}^{(1)} - d_{i+1}| \}
+ \frac{1}{384 c} \{ h^4 \| f^{(4)} \|_{\infty} (1 + \frac{1}{4}\underset{i \in  J} \max|r_i-3|) \\
 & + 4 ~\underset{i \in  J} \max|r_i-3| \big ( h^3  \| f^{(3)} \|_{\infty}  + 3 h^2  \| f^{(2)} \|_{\infty} \big ) \big \} + \frac {  |\alpha|_{\infty} (  \|s \|_{\infty} +  Z_0)}{ 1 -|\alpha|_{\infty}
 },
\end{split}
\end{equation}
where $|\alpha|_\infty=\max\{\alpha_i: i \in J\}$, $ h = \max \{
h_i : i \in J \} $, and $ c = \min\{ c_i : i \in J \}$ with
\begin{equation}\label{r15}
 c_i = \left\{\begin{array}{rlll}
& \frac{1+ r_i}{4}, \; -1 < r_i <  3,\\
& 1, \; \; \; \; \quad r_i \ge 3.
\end{array}\right.
\end{equation}
\end{thm}
\begin{proof}
For a prescribed set of data and  $\alpha_i$ satisfying $
|\alpha_i| \le \kappa a_i,  i \in J $, the rational cubic spline
FIF $S \in$ $\mathcal{C}^1(I)$ is the fixed point of the
Read-Bajraktarevi\'{c} operator $T_{\alpha}$:
 \begin{equation}\label{r13}
  (T_{\alpha}S)(x) = \alpha_i S(L_i^{-1}(x)) + R_i (\alpha_i,
  \phi),
 \end{equation}
where $ R_i (\alpha_i, \phi) = \frac{P_i^*(\alpha_i,
\phi)}{Q_i^*(\phi)}$, $\phi = \frac{L_i^{-1}(x) - x_1}{x_N -
x_1}=\frac{x - x_i}{h_i} , \; x \in [x_i, x_{i+1}], \; i \in J,$
with $P_i ^*$ and $Q_i^*$  as in (\ref{r10}). Note that the
subscript $\alpha$ is used to emphasize the dependence of the map
$T$ on the scale vector $\alpha$. The coefficients of the rational
function $R_i$ depend on the scaling factor $\alpha_i$, and hence
$R_i$ can be thought of as a function of $\alpha_i$ and $\phi$.
The interpolants $S$ and $s$ are fixed points of $T_\alpha$ with
$\alpha \ne {\bf 0}$ and $ \alpha = {\bf 0} $ respectively. We
know  \cite{DG4} that for $ x \in [x_i, x_{i+1}]$,
\begin{equation}\label{r14}
\begin{split}
|f(x) - s(x) | \le & \frac{h_i}{4c_i} \max \Big \{ |y_i^{(1)} -
d_i|,  |y_{i+1}^{(1)} - d_{i+1}|\Big \}
+ \frac{1}{384 c_i} \Big\{ h_i^4 \| f^{(4)} \| (1 + \frac{|r_i -3| }{4}) \\
 & + 4 |r_i - 3| \big ( h_i^3  \| f^{(3)} \|  + 3 h_i^2  \| f^{(2)} \|\big ) \Big
 \},
\end{split}
\end{equation}
where $ y_i^{(1)} = f^{(1)} (x_i)$ for $i = 1,2, \dots, N,$ and
$\|.\|$ denotes the uniform norm on $[x_i,x_{i+1}]$.\\ For a fixed
choice of scale vector $\alpha \ne {\bf 0}$ and for $ x \in [x_i,
x_{i+1}]$, from (\ref{r13}) we obtain:
\begin{equation*}
\begin{split}
 |T_{\alpha} S(x) - T_{\alpha} s(x) |  & = \Big|\big \{ \alpha_i  S(L_i^{-1}(x)) + R_i (\alpha_i, \phi )\big\} - \big\{ \alpha_i s(L_i^{-1}(x)) + R_i (\alpha_i, \phi )\big\} \Big|, \\
& \le |\alpha|_{\infty} \|S - s \|_{\infty}.
\end{split}
\end{equation*}
From the above inequality we deduce:
\begin{equation}\label{r16}
\|T_{\alpha} S- T_{\alpha} s \|_{\infty} \le  |\alpha|_{\infty}
\|S - s \|_{\infty}.
\end{equation}
Let $ x \in [x_i, x_{i+1}]$  and  $\alpha \ne {\bf 0}$. Using
(\ref{r13}) and the Mean Value Theorem:
\begin{equation*}
\begin{split}
 |T_{\alpha} s(x) - T_{\bf 0} s(x) |  & = \Big|\big \{ \alpha_i  s(L_i^{-1}(x)) + R_i (\alpha_i, \phi )\big\} -  R_i (0, \phi)  \Big|, \\
& \le | \alpha_i| \|s \|_{\infty} + | \alpha_i|  \Big |
\frac{\partial R_i(\tau_i, \phi)}{\partial \alpha_i} \Big |, \\&
\le |\alpha_i| ( \|s \|_{\infty} + Z_0),
\end{split}
\end{equation*}
Thus,
\begin{equation}\label{r17}
 \|T_{\alpha} s- T_{\bf 0} s \|_{\infty} \le |\alpha|_{\infty}  ( \|s \|_{\infty} +
 Z_0).
\end{equation}
Using  (\ref{r16}) and (\ref{r17}),
\begin{equation*}
\begin{split}
 \|S - s \|_{\infty} = \| T_{\alpha} S - T_{\bf 0} s \|_{\infty} &  \le  \| T_{\alpha} S - T_{\alpha} s \|_{\infty} + \| T_{\alpha} s - T_{\bf 0} s \|_{\infty}, \\
& \le |\alpha|_{\infty} \| S - s \|_{\infty} + |\alpha|_{\infty} (
\|s \|_{\infty} +  Z_0),
\end{split}
\end{equation*}
which simplifies to
\begin{equation}\label{r20}
 \|S - s \|_{\infty} \le  \frac {  |\alpha|_{\infty} (  \|s \|_{\infty} +  Z_0)}{ 1 -|\alpha|_{\infty}
 }.
\end{equation}
 The required assertion follows  from (\ref{r14a}), (\ref{r14}), and
(\ref{r20}). However, in what follows, we find an upper bound for
$ \|s \|_{\infty}$ and estimate $Z_0$, if not optimally, at least
practically.\\
Let us introduce the notations: $ |y|_{\infty} = \max \{|y_i| : 1
\le i \le N \} $, $ |d|_{\infty} = \max \{|d_i| : 1 \le i \le N \}
$, and $ |r|_{\infty} = \max \{ |r_i| :  i \in J \} $. From
(\ref{r11a}), for $ x \in [x_1, x_N]$,
$$|s(x)| \le \frac {\max \{|P_i^{**}(\phi) | : i \in J , 0 \le \phi \le 1 \}}{\min \{| Q_i^*(\phi)| : i \in J , 0 \le \phi \le 1\}}, $$
where $P_i^{**} (\phi)$ is the numerator in (\ref{r11a}). Using
extremum calculations of polynomials, {\small
\begin{equation*}
 \begin{split}
 |P_i^{**}(\phi) |  \le | y_i|  (1-\phi)^3 & + (|r_i| | y_i| + h_i |d_i|)  \phi (1- \phi)^2 + ( |r_i| | y_{i+1}| + h_i | d_{i+1}|) \phi^2 (1-\phi) +| y_{i+1}| \phi^3,\\
\implies \underset{\phi \in [0,1]} \max |P_i^{**}(\phi) |  & \le \max\{| y_i|,| y_{i+1}|\} + \frac{1}{4}  \Big(|r_i| \max\{| y_i|,| y_{i+1}|\} + h_i \max\{|d_i|, | d_{i+1}|\}\Big), \\
\implies \underset{ i \in J, \phi \in [0,1]} \max |P_i^{**}(\phi)
| & \le  |y|_{\infty} + \frac{1}{4} ( |r|_{\infty} |y|_{\infty} +
h |d|_{\infty}),
 \end{split}
\end{equation*} }
and $|Q_i^* (\phi)| = Q_i^* (\phi)  \ge c_i$. Therefore,
\begin{equation}\label{r18}
 \|s \|_{\infty} \le \frac { |y|_{\infty} + \frac{1}{4} (  |r|_{\infty} |y|_{\infty} + h  |d|_{\infty}) } { \min\{ c_i : i \in J \}
 }.
\end{equation}
Now, from  (\ref{r10}) and  (\ref{r13}), for $x \in [x_i, x_{i+1}]$,
\begin{equation*}
 \frac{\partial R_i}{\partial \alpha_i} = \frac{ \tilde{P_i}(\phi)}
 {Q_i^*
 (\phi)},
\end{equation*}
where $\tilde{P_i}(\phi) =   - \{ y_1 (1-\phi)^3 + (r_i y_1+ (x_N
- x_1) d_1)  \phi (1- \phi)^2 + (r_i y_N -  d_N  (x_N - x_1))
\phi^2 (1-\phi) + y_N \phi^3 \}$. Using similar extremum
calculations,
\begin{equation}\label{r19}
 \Big | \frac{\partial R_i}{\partial \alpha_i} \Big |  \le \;  Z_0 = \frac {\max\{|y_1|, |y_N|\}(1 + \frac{1}{4} |r|_{\infty}) + \frac{1}{4}|I| \max\{ |d_1|,|d_N|\} } { \min\{ c_i : i \in J \}
 }~ \forall~ i \in J,
\end{equation}
where $ |I| = x_N - x_1$.
\end{proof}
\noindent Due to the principle of construction of a smooth FIF,
for $S$ to be in the class $\mathcal{C}^1(I)$, we impose $
|\alpha_i| \le \kappa a_i = \frac{\kappa h_i}{x_N - x_1}$. Hence,
$ |\alpha|_{\infty} \le \frac{\kappa h} {x_N - x_1}$, and
consequently $S$ converges uniformly to the original function when
the norm of the partition tends to zero. The following convergence
results are direct consequences of Theorem \ref{t1}.
\begin{cor}\label{t2}
 Let $S$ be the rational FIF with respect to the data points $\{(x_i, y_i) : i = 1,2, \dots, N \}$ corresponding
 to the original function $ f \in$ $\mathcal{C}^4(I)$.
 Suppose $d_i, i = 1,2, \dots, N$, and
$ r_i, \alpha_i, i \in J$ are chosen accordingly.
\begin{enumerate}[(i)]
 \item  If $ y_i^{(1)} - d_i = O(h_i) = y_{i+1}^{(1)} - d_{i+1}$, $r_i > -1$ and $|\alpha_i| < a_i^2 $, then $  \|f - S \|_{\infty} = O (h^2)$.
\item  If $ y_i^{(1)} - d_i = O (h_i^2) = y_{i+1}^{(1)} -
d_{i+1}$, $ r_i - 3 = O(h_i)$ and $|\alpha_i| < a_i^3 $, then $
\|f - S \|_{\infty} = O (h^3)$. \item  If $ y_i^{(1)} - d_i = O
(h_i^3) = y_{i+1}^{(1)} - d_{i+1}$,  $ r_i - 3 = O(h_i^2)$ and
$|\alpha_i| < a_i^4$, then $  \|f - S \|_{\infty} = O (h^4)$.
\end{enumerate}
\end{cor}
\noindent The above theorem and corollary show that $r_i$ should
ideally be such that $r_i-3=O(h_i^2)$. Later we shall consider how
$r_i$ can be chosen to preserve the data monotonicity/convexity,
whilst maintaining this optimal requirement.\\ Following the
convergence results for the classical rational cubic spline $s$
studied by  Gregory and Delbourgo, we assumed that the data
generating function $f$ is in class $\mathcal{C}^4(I)$. Now we
establish the uniform convergence of $s$ with a weaker assumption
$f\in$ $\mathcal{C}^1(I)$, and use it to deduce the uniform
convergence of rational cubic spline FIF $S$ as in the previous
case.
\begin{thm}\label{t1a}
Let $S$ and $s$, respectively, be the rational cubic spline FIF
and the classical rational cubic spline interpolant for the
original function $ f \in$ $\mathcal{C}^1(I)$ with respect to the
interpolation data $\{(x_i, y_i) : i = 1,2, \dots, N \}$. Suppose
that the rational function $R_i$  involved in the IFS generating
the FIF $S$ satisfies $ \big| \frac{\partial R_i(\tau_i,
\phi)}{\partial \alpha_i} \big| \le Z_0 $, $|\tau_i| \in (0,
\kappa a_i)$, for all $i \in J$, and for some real constant $Z_0$.
Then,
\begin{equation*}
\|f - S \|_{\infty} \le \frac{1}{4c} h |d|_\infty + \frac{1}{4c}
\omega(f;h) (|r|_\infty +4)    + \frac { |\alpha|_{\infty} ( \|s
\|_{\infty} + Z_0)}{ 1 -|\alpha|_{\infty}
 },
\end{equation*}
where $ c = \min\{ c_i : i \in J \}$, and $\omega(f;h):=
{\underset {|x - x^*| \le h}  \sup}\big \{ |f(x)-f(x^*)| : x, x^*
\in I \big\} $  is the modulus of continuity of $f$. In
particular, $S$ converges uniformly to $f\in$ $\mathcal{C}^1(I)$
as the mesh norm tends to zero.
\end{thm}
\begin{proof}
 Observe that
\begin{equation*}
\begin{split}
Q_i(\theta) &= 1+ (r_i-3) \theta (1-\theta),\\
&= (1-\theta)^3 + r_i \theta (1-\theta)^2 + r_i \theta^2
(1-\theta)+ \theta^3.
\end{split}
\end{equation*}
For $x \in [x_i, x_{i+1}]$,
\begin{equation*}
\begin{split}
f(x) - s(x)=& f(x)- \frac{P_i ^* (\theta)}{Q_i (\theta)}\\
=& \frac{1}{1+(r_i-3)\theta(1-\theta)}\Big[(1-\theta)^3
\big(f(x)-y_i\big) + r_i \theta (1-\theta)^2
\big(f(x)-y_i\big)+\\& r_i \theta^2
(1-\theta)\big(f(x)-y_{i+1}\big)+ \theta^3
\big(f(x)-y_{i+1}\big)-h_i d_i \theta (1-\theta)^2 + h_i d_{i+1}
\theta^2 (1-\theta)\Big].
\end{split}
\end{equation*}
Therefore, local error of the interpolation is given by
\begin{equation*}
\begin{split}
|f(x)-s(x)| \le & \frac{1}{c_i}\Big\{
|f(x)-y_i|[(1-\theta)^3+|r_i| \theta(1-\theta)^2]+ |f(x)
-y_{i+1}|[|r_i| \theta^2 (1-\theta) + \theta^3]\\ &+ h_i [|d_i|
\theta (1-\theta)^2 + |d_{i+1}| \theta^2
(1-\theta)]\Big\},\\
 \le & \frac{1}{c_i}\Big[(\frac{1}{4}|r_i| +1)\omega(f;h) +
\frac{1}{4} h_i \max \{|d_i|, |d_{i+1}|\}\Big],
\end{split}
\end{equation*}
Consequently,  we have the following uniform error bound for the
classical rational spline $s$:
\begin{equation} \label{r21a}
\begin{split}
\|f-s\|_\infty & \le \frac{1}{4c} h |d|_\infty + \frac{1}{4c}
\omega(f;h) (|r|_\infty +4).
\end{split}
\end{equation}
Now (\ref{r14a}) coupled with (\ref{r20}) and (\ref{r21a}) proves
the  theorem.
\end{proof}
\section{Shape Preserving Rational Fractal Interpolation}\label{r1ps4}
\setcounter{equation}{0}
\subsection{Monotonic Data}\label{r1ps41}
For the sake of simplicity, let us assume that the given data set
is monotonically increasing, i.e., $y_1 \le y_2 \le \dots \le
y_N$, and consequently $ \Delta_i = \frac{y_{i+1} - y_i}{h_i} \ge
0 ~\forall~ i\in J$. For a monotonic increasing interpolant $S\in
\mathcal{C}^1$, it is necessary that  the derivative parameters
satisfy $d_i \ge 0 , i = 1,2, \dots,N $. From elementary calculus,
we know that a differentiable function $S$ is monotonic increasing
on $I$ if and only if $S^{(1)}(x) \ge 0 $ for all $ x \in I$.
Calculation of $S^{(1)}(L_i(x))$ from (\ref{r10}) and further
simplification give: for $ x\in I$,
\begin{equation}\label{r22}
 S^{(1)}\big(L_i(x)\big) = \frac{\alpha_i}{a_i} S^{(1)}(x) + \frac{ T_i \theta^4 + S_i \theta^3 (1-\theta) + U_i  \theta^2 (1-\theta)^2 + V_i \theta (1-\theta)^3 + W_i (1-\theta)^4}{[ 1 + (r_i -3 ) \theta (1-\theta)]^2
 },
\end{equation}
\begin{equation*}
\begin{split}
\text{where}\;T_i & = d_{i+1} - \frac{\alpha_i}{h_i} (x_N - x_1) d_N ,\\
S_i & = 2(r_i \Delta_i - d_i) - \frac{2\alpha_i}{h_i} [r_i (y_N - y_1) - d_1 (x_N - x_1)] ,\\
U_i & = (r_i^2 + 3) \Delta_i  - r_i (d_i + d_{i+1}) - \frac{\alpha_i}{h_i} [  (r_i^2 + 3)  (y_N - y_1) - r_i  (x_N - x_1) (d_1 + d_N) ],\\
V_i & = 2(r_i \Delta_i - d_{i+1}) - \frac{2\alpha_i}{h_i} [r_i (y_N - y_1) - d_N (x_N - x_1)],\\
W_i & = d_i - \frac{\alpha_i}{h_i} (x_N - x_1) d_1.
\end{split}
\end{equation*}
Since the rational cubic spline FIF is defined implicitly and
recursively, to maintain the positivity of $S^{(1)}$ in the
successive iterations and to keep the desired data dependent
monotonicity condition to be simple enough, we assume $\alpha_i
\ge 0$ for all $i \in J$. Our predilection to the nonnegativity
assumption on the scaling factors is attributable to reasons of
convenience rather than of necessity.
 Then, for $i \in J$ and an arbitrary knot point $x_j$, sufficient
conditions for $S^{(1)}\big(L_i(x_j)\big)\ge 0$ are:
\begin{equation}
 T_i \ge 0 , S_i \ge 0, U_i \ge 0 , V_i \ge 0, W_i \ge
 0,
\end{equation}
where the necessary condition on the derivative parameters are assumed. \\
Now $ T_i \ge 0  \Leftrightarrow \frac{\alpha_i}{h_i} (x_N - x_1)
d_N \le d_{i+1}$. Observe that  if $d_N =0$, then $T_i \geq 0 $,
$i \in J$, follow  directly from the assumption on the derivative
parameters. Otherwise, we impose the following condition on the
scaling factors:
\begin{equation}\label{r24}
 \alpha_i \le \frac {d_{i+1} h_i} { d_N (x_N - x_1)}.
\end{equation}
 \noindent Similarly $ W_i \ge 0 $, whenever the scaling factor satisfies
\begin{equation}\label{r25}
 \alpha_i \le \frac {d_i h_i} { d_1 (x_N - x_1)}.
\end{equation}
Again
\begin{equation}\label{r26}\left.
\begin{split}
 S_i \ge 0 &\Leftrightarrow  h_i (r_i  \Delta_i - d_i ) \ge \alpha_i [ r_i (y_N - y_1)  -d_1  (x_N - x_1)], \\
 V_i \ge 0 & \Leftrightarrow h_i (r_i  \Delta_i - d_{i+1}) \ge \alpha_i [ r_i (y_N - y_1)  -d_N  (x_N -
 x_1)].
\end{split}\right\}
\end{equation}
If it is  assumed  that
\begin{equation}\label{r27}
 r_i [ h_i \Delta_i - \alpha_i (y_N - y_1)]\ge h_i (d_i + d_{i+1}) - \alpha_i (x_N - x_1) (d_1 +
 d_N),
\end{equation}
then it follows from (\ref{r24})-(\ref{r25}) that
\begin{equation*}
 \begin{split}
   r_i [ h_i \Delta_i - \alpha_i (y_N - y_1) ] & \ge h_i  d_i  - \alpha_i d_1 (x_N - x_1) \ge 0, \\
   r_i [ h_i \Delta_i - \alpha_i (y_N - y_1) ] & \ge h_i d_{i+1} - \alpha_i d_N  (x_N - x_1) \ge 0.
 \end{split}
\end{equation*}
In view of (\ref{r26}), the above inequalities  imply  $S_i \ge 0
$ and $V_i \ge 0$.\\ Assume that $ h_i \Delta_i - \alpha_i
(y_N-y_1) \ge 0$, i.e.,
\begin{equation}\label{r29}
 \alpha_i \le \frac{ h_i \Delta_i}{y_N - y_1}.
\end{equation}
From (\ref{r27}), we have
\begin{equation}\label{r28}
 r_i^2 [ h_i \Delta_i - \alpha_i (y_N - y_1) ]   \ge r_i h_i (d_i + d_{i+1}) - r_i \alpha_i (x_N - x_1) (d_1 +
 d_N).
\end{equation}
From (\ref{r28}) and $ 3 h_i  \Delta_i \ge 3 \alpha_i (y_N -
y_1)$, it is easy to verify that $U_i \ge 0$. From (\ref{r24}),
(\ref{r25}) and (\ref{r29}), for a monotonicity preserving
rational FIF, it suffices to choose $\alpha_i, i \in J$, according
to:
\begin{equation}\label{r30}
 0 \le \alpha_i \le \min \Big \{ \kappa a_i,  \frac {d_i h_i} { d_1 (x_N - x_1)},    \frac {d_{i+1} h_i} { d_N (x_N - x_1)},   \frac{ \Delta_i h_i }{y_N - y_1}    \Big
 \}.
\end{equation}
Here the first term in the braces arises due to the
$\mathcal{C}^1$-continuity of $S$. After choosing the scaling
factor $\alpha_i$ according to (\ref{r30}), the shape parameters
$r_i$ is selected to fulfill inequality (\ref{r27}), and these
conditions are sufficient for $S^{(1)}\big(L_i(x_j)\big)\ge 0$. As
$[x_1, x_N]$ is the attractor of the IFS $\{\mathbb{R};L_i(x) ,
i\in J \}$, by the recursive nature of the rational fractal
function, $S^{(1)} \big(L_i(x_j)\big) \ge 0 $ for all $i \in J $
and for every knot point $x_j$ imply that  $ S^{(1)}(x) \ge 0 $
for all $ x \in I.$
\par
\noindent With the necessary condition $d_i \le 0, i \in J $,
assumed to hold, an analogous procedure applies for a monotonic
decreasing data set.
\begin{rem} \label{rcon}
  If $\Delta_i = 0$, then we take $\alpha_i =0$ for the monotonicity
 of the rational cubic spline  FIF. Also in this case, $d_i= d_{i+1}=0$. Consequently, $S\big(L_i(x)\big) = y_i
 = y_{i+1}$ ,i.e., to say that  $S$ reduces to a constant on the
 interval $[x_i, x_{i+1}]$.
\end{rem}
\begin{rem} \label{GDM}
When all $\alpha_i = 0$, the rational cubic FIF reduces to the
classical rational cubic spline $s$. In this case,  condition
(\ref{r30}) is obviously true, and the condition (\ref{r27})
reduces to
\begin{equation}\label{r31}
 r_i   \ge  (d_i + d_{i+1})/   \Delta_i,\; i \in J.
\end{equation}
 Thus (\ref{r31}) is a sufficient condition for the monotonicity
of the classical rational cubic spline $s$ [\cite{DG4}, p. 970].
\end{rem}
\begin{rem}
 For a given  strictly  monotonic data,  we select $\alpha_i$ satisfying  (\ref{r30}) with $\alpha_i < \frac{h_i
\Delta_i}{y_N-y_1}$,  and then fix the shape parameters according
to
\begin{equation}\label{30}
r_i =  1 + \frac{h_i(d_i + d_{i+1}) - \alpha_i (x_N - x_1) (d_1 +
d_N)}{h_i \Delta_i - \alpha_i (y_N - y_1)},
\end{equation} so  that the monotonicity condition (\ref{r27}) is satisfied. With these choices of the IFS parameters, the rational cubic spline FIF reduces to the  rational quadratic FIF:\\
\begin{equation}\label{r30a}
S\big(L_i(x)\big) = \alpha_i S(x)  + \frac{
P_i(\theta)}{Q_i(\theta) },
\end{equation}
where $ P_i(\theta) =  (y_{i+1} - \alpha_i y_N) \Delta_i \theta^2 +   \beta_{i}\Big \{ (y_i d_{i+1}+ y_{i+1}d_i)a_i  -  \alpha_i \big [y_{i+1}d_1 +y_i d_N+a_i (y_N d_i+ y_1 d_{i+1})\big] $\\
 $+  \alpha_{i}^{2} (y_N d_1+ y_1 d_N) \Big\} \theta (1-\theta) + (y_i -\alpha_i y_1)  \Delta_i (1-\theta)^{2}$,\\
$Q_i(\theta) = \Delta_i \theta^2 + \beta_i \big[a_i  (d_i + d_{i+1}) - \alpha_i  (d_1 + d_N)\big] \theta (1-\theta) + \Delta_i (1-  \theta)^2$, and \\
 $\beta_{i}=\dfrac{\Delta_i (x_N-x_1)} { (y_{i+1}-y_i )-\alpha_i (y_N-y_1)}$.
 For $\Delta_i =0$, we choose $\alpha_i =0$, and define $S(L_i(x))= y_i= y_{i+1}$.
 If all $\alpha_i = 0$  in (\ref {r30a}), then the corresponding  rational quadratic FIF reduces to
the classical monotonic rational quadratic interpolant studied in
detail in \cite{DG1}. In this degenerated case, the necessary
condition $d_i\geq 0$ is also sufficient for
 the monotonicity of the interpolant (see \cite{DG1}).
\end{rem}
\begin{rem}\label{remb}
For the shape parameters specified in (\ref{30}) and $|\alpha_i| <
a_i^4$, we get $r_i-3 = O(h_i^2)$. Consequently, from corollary
3.1 it follows that (\ref{30}) is a good choice of $r_i$ since the
optimal $O(h^4)$ bound on the interpolation error can be achieved.
 Hence, for a monotonic rational cubic spline FIF with an optimal error bound,
 we choose
 $0 \le \alpha_i \le \min \Big \{  \kappa
a_i^4, \frac {d_i a_i}{d_1}, \frac {d_{i+1} a_i} { d_N }, \kappa^*
\frac{ \Delta_i h_i }{y_N - y_1} \Big \}$,  $\kappa, \kappa^* \in
[0,1),$ and $r_i$ as in (\ref{30}).
\end{rem}
\noindent The entire discussion on the monotonic rational cubic
spline FIFs can be encapsulated in the following theorem:
\begin{thm} \label {t2}
 For a given set of monotonic data $\{(x_i, y_i) :
 i=1,2,\dots, N\}$, let $S$ be the rational cubic spline FIF
 described  in (\ref{r10}). Assume that the necessary conditions
 on the derivative parameters are satisfied. Then, the following
 conditions on the scaling factors and the shape parameters on each subinterval are
 sufficient for $S$ to be monotonic on  $I= [x_1, x_N]:$
$$0 \le \alpha_i \le \min \Big \{  \kappa a_i,  \frac {d_i
a_i}{d_1}, \frac {d_{i+1} a_i} { d_N },  \kappa^* \frac{ \Delta_i
h_i }{y_N - y_1} \Big \};\;  \kappa, \kappa^* \in [0,1),$$
$$ r_i \geq \frac{h_i(d_i + d_{i+1}) - \alpha_i (x_N - x_1) (d_1 + d_N)}{h_i \Delta_i - \alpha_i (y_N - y_1)}.$$
In particular, if $r_i =  1 + \dfrac{h_i(d_i + d_{i+1}) - \alpha_i
(x_N - x_1) (d_1 + d_N)}{h_i \Delta_i - \alpha_i (y_N - y_1)},$
then $S$ reduces to the rational quadratic spline FIF given in
(\ref {r30a}), and in this case the above mentioned conditions on
$\alpha_i$ alone are sufficient for the rational FIF to be
monotonic.
\end{thm}
\subsection{Convex Data}\label{r1ps42}
We assume a strictly convex set of data so that:
\begin{equation}\label{cc}
 \Delta_1 < \Delta_2 < \dots < \Delta_{N-1}.
\end{equation}
To  have a convex interpolant $S$ and to avoid the possibility of
$S$ having straight line segments, it is necessary that  the
derivatives  at knot points satisfy
\begin{equation}\label{cdc}
 d_1 < \Delta_1 < d_2 < \Delta_2 < \dots < d_i < \Delta_i < \dots <
 d_N.
\end{equation}
For a concave data set  inequality will be reversed. Let $s$ be
the classical counterpart of $S$ studied elaborately  in
\cite{DG4}. Since $s$ may fail to have second derivative at knot
points, $s$ is not twice differentiable on the entire interval
$I$. Hence, in contrast to the claim  made in \cite{DG4},
convexity of $s$ cannot be derived from the result that reads:
\emph{$s(x)$ is convex if and only if $s^{(2)}(x) \ge 0$ for all
$x \in I$}. However, the following result from elementary calculus
justify the procedure adapted in \cite{DG4}: \emph{Suppose that
$f$ is piecewise $\mathcal{C}^2$ with increasing slopes, i.e.,
there is a subdivision $a=x_0<x_1<\dots<x_k=b$ of $I=[a,b]$ such
that (i) $f$ is continuous on $I$ (ii) $f$ is of class
$\mathcal{C}^2$ on each subinterval $(x_{i-1},x_i)$, $i=1,2,\dots,
k$ (iii) $f$ has one-sided derivative at $x_1,x_2,\dots,x_{k-1}$
satisfying $f(x_i^-) \le f(x_i^+)$ for $i=1,2,\dots, k-1$, then
$f$ is convex
on $I$}.\\
Now turning our attention to the rational cubic spline FIF $S$, it
is worth mentioning that due to the fractality, $S$ may not be
even piecewise $\mathcal{C}^2$. Consequently, we cannot adapt the
convexity procedure for the classical cubic spline $s$ by applying
the result stated above. Instead, we use the following results:\\
(i) \emph{A differentiable function of one variable is convex on
an interval if and only if its derivative is monotonically
increasing on that interval.}\\ (ii) \emph{Let $f$ be a continuous
function on $[a,b]$. If for each $x \in (a,b)$ one of the one
sided derivatives $f^{(1)}(x^+)$ or $f^{(1)}(x^-)$ exists and is
nonnegative (possibly $+\infty$), then $f$ is monotonic increasing
on $[a,b]$.}\\
Owing to these results, to establish the convexity of $S \in
\mathcal{C}^1(I)$, it is enough to show that $S^{(2)}(x^+)$
or $S^{(2)}(x^-)$ exists and is nonnegative for each $x \in (x_1,x_N)$. It is to this that we now turn.\\
Informally, \begin{equation}\label{r32}
 S^{(2)}\big(L_i(x)\big) = \frac{\alpha_i}{a_i^2} S^{(2)}{(x)} +
 R_i^{(2)}(x),
 \end{equation}
 \begin{equation*} \text{where}\; R_i^{(2)}(x)= \frac{ 2 [ A_i^* \theta^3 +
B_i^* \theta^2 (1-\theta) + C_i^* \theta (1-\theta)^2 + D_i^*
(1-\theta)^3 ]}{h_i [ 1 + (r_i -3 ) \theta (1-\theta)]^3 },
\end{equation*}
\begin{equation*}\label{r33}\left.
\begin{split}
A_i^* & = r_i (d_{i+1} - \Delta_i) + d_i - d_{i+1} - \frac{\alpha_i}{h_i} \big\{ r_i [ d_N (x_N - x_1) - (y_N - y_1)] + (x_N - x_1) (d_1 - d_N) \big \} ,\\
B_i^* & =  3 ( d_{i+1} -\Delta_i) - \frac{3\alpha_i}{h_i}    [ d_N (x_N - x_1) - (y_N - y_1)]    , \\
C_i^* & =  3 ( \Delta_i - d_i)   - \frac{3 \alpha_i}{h_i}  [ (y_N - y_1) - d_1 (x_N - x_1) ]    , \\
D_i^* & =  r_i (\Delta_i - d_i ) + d_i - d_{i+1} - \frac{\alpha_i}{h_i} \big\{ r_i [   (y_N - y_1)-d_1 (x_N - x_1) ] + (x_N - x_1) (d_1 - d_N) \big \} .\\
\end{split} \right\}
\end{equation*}
We assume that $0 \le \alpha_i \le \kappa a_i^2 $ for $i \in J$,
where $0 \le \kappa < 1$. Using the fact that
$L_j:[x_1,x_N]\rightarrow [x_j, x_{j+1}]$ satisfies
$L_j(x_1)=x_j$, $L_j(x_N)=x_{j+1}$, for $j \in J$ we get:
\begin{equation}\label{r34a}\left.
\begin{split}
S^{(2)}(x_1^+) &= \frac{2 D_1^*}{h_1} \big [1-
\frac{\alpha_1}{a_1^2} \big]^{-1},\\
S^{(2)}(x_j ^+) &= \frac{\alpha_j}{a_j^2} S^{(2)} (x_1^+) +
\frac{2 D_j^*}{h_j};\; j=2,3, \dots, N-1,\\
 S^{(2)}(x_N^-) &=
\frac{2 A_{N-1}^*}{h_{N-1}} \big [1-
\frac{\alpha_{N-1}}{a_{N-1}^2} \big]^{-1}.
\end{split}\right\}
\end{equation}
From (\ref{r34a}), it follows that the second derivative
(right-hand) at the knot points $x_j$, $j \in J,$ and the second
derivative (left-hand) at the extreme end point $x_N$ is
nonnegative if: $D_j^*$, $j \in J$, and $A_{N-1}^*$ are
nonnegative. For a typical knot point $x_j$, $j \in J:$
\begin{equation}\label{r34b}
S^{(2)}\big(L_i(x_j)^+\big)= \frac{\alpha_i}{a_i^2}S^{(2)}(x_j^+)
+ R_i^{(2)}(x_j^+)
\end{equation}
Assuming $D_j^*$, $j \in J$, to be nonnegative, (\ref{r34b})
suggests that $S^{(2)}\big(L_i(x_j)^+\big) \ge 0$ is satisfied,
provided $R_i^{(2)}(x_j^+) \ge 0$. Again, $R_i^{(2)}(x_j^+) \ge 0$
is satisfied if:
$$ A_i ^* \ge 0 , B_i^* \ge 0,  C_i^* \ge 0,\text{and}~  D_i^* \ge 0;\; i \in J. $$
From the Three Chords Lemma for convex functions \cite{LB}, it follows that a convex set of data should necessarily satisfy $ d_1 \le \frac{y_N - y_1}{x_N - x_1} \le d_N$, where inequalities remain strict for strict convexity.\\
Now $  B_i^* \ge 0  \Leftrightarrow d_{i+1} - \Delta_i  \ge
\frac{\alpha_i}{h_i}    [ d_N (x_N - x_1) - (y_N - y_1)]$.
Observing that  if $d_N(x_N-x_1)- (y_N-y_1)=0$, then $B_i^* \ge 0
$ is obviously satisfied, we get the condition on the scaling
factor as
\begin{equation*}\label{r34}
 \alpha_i \le \frac{h_i (d_{i+1} - \Delta_i)}{d_N (x_N - x_1) - (y_N - y_1)}.
\end{equation*}
Similarly, $ C_i^* \ge 0  \Leftrightarrow \Delta_i - d_i \ge \frac{\alpha_i}{h_i}  [ (y_N - y_1) - d_1 (x_N - x_1) ] $ gives
\begin{equation*}\label{r35}
 \alpha_i \le \frac{h_i ( \Delta_i - d_i )}{ (y_N - y_1) -d_1 (x_N - x_1)}.
\end{equation*}
Therefore, to obtain $S^{(2)} (L_i(x_j)^+) \ge 0$ for all $i \in
J$ and knot points $x_j, j\in J$, it suffices to have  $ A_i^* \ge
0, D_i^* \ge 0$, and
\begin{equation}\label{r36}
 0 \le \alpha_i  \le \min \Big \{ \kappa a_i ^2, \frac{h_i (d_{i+1} - \Delta_i)}{d_N (x_N - x_1) - (y_N - y_1)},    \frac{h_i ( \Delta_i - d_i )}{ (y_N - y_1) -d_1 (x_N - x_1) } \Big \}.
\end{equation} The following conditions on the shape parameter $r_i$ give  $  A_i^* \ge 0, D_i^* \ge 0$.
{\scriptsize \begin{equation*}
 r_i \ge \max \Big\{ \frac{d_{i+1}- d_i + (\alpha_i/h_i)   (x_N - x_1)  (d_1 - d_N)} { d_{i+1} - \Delta_i - (\alpha_i/h_i) [ d_N (x_N - x_1) - (y_N - y_1)]},
\frac {d_{i+1}- d_i + (\alpha_i/h_i)   (x_N - x_1)  (d_1 - d_N)} {
\Delta_i - d_i -(\alpha_i/h_i) [ (y_N - y_1) - d_1 (x_N -
x_1)]}\Big \}
\end{equation*}}
The condition on $r_i$ stated above is equivalent to
\begin{equation}\label{r37}
 r_i \ge  1 + \frac{M_i}{m_i},
\end{equation}
where
$$ M_i = \max \{   d_{i+1} - \Delta_i -\frac{\alpha_i}{h_i} [ d_N (x_N - x_1) - (y_N - y_1)] ,  \Delta_i - d_i  -\frac{\alpha_i}{h_i}  [ (y_N - y_1) - d_1 (x_N - x_1)] \},$$
$$m_i = \min  \{ d_{i+1} - \Delta_i -\frac{\alpha_i}{h_i} [ d_N (x_N - x_1) - (y_N - y_1)] ,  \Delta_i - d_i  -\frac{\alpha_i}{h_i}  [(y_N - y_1) - d_1 (x_N - x_1)] \}.$$
Therefore the  conditions (\ref{r36}) on the scaling factors and
(\ref{r37}) on the shape parameters ensure
$S^{(2)}\big(L_i(x_j)^+\big) \ge 0$ for all $i, j \in J$ and
$S^{(2)}(x_n^-) \ge 0$. Since the rational fractal function is
generated recursively and $[x_1, x_N]$ is the attractor of the IFS
$\{R; L_i(x): i \in J\}$ $S^{(2)}\big(L_i(x_j)^+\big) \ge 0$ for
all $i, j \in J$ yield $S^{(2)}(x^+) \ge0$ for all $x \in
(x_1,x_N)$. Hence, by the result quoted at the beginning of this
section $S^{(1)}$ is monotonically increasing, as a consequence of
which $S$ is convex. Analogous procedure applies to a concave data
set.
\begin{rem}
If $ \Delta_i = \Delta_{i+1}$, then for a convex fractal
interpolant, we choose the scaling factor $\alpha_i$ to be zero.
Also, $d_i = d_{i+1}= \Delta_i$. Thus, in this case the rational
cubic FIF becomes $S(x) = \frac{(x_{i+1}-x)y_{i}
+(x-x_i)y_{i+1}}{x_{i+1}-x_i}$, i.e., to say that $S$ reduces to a
straight line segment on $[x_i, x_{i+1}]$, as would be expected.
\end{rem}
\begin{rem}
When all $\alpha_i = 0$, the condition (\ref{r36}) is obviously
true and the condition (\ref{r37}) reduces to
\begin{equation}\label{r38b}
\begin{split}
 r_i  & \ge  1+ \frac{M_i^*}{m_i^*}\quad \text{with}\\
 M_i^* &=\max\{d_{i+1}-\Delta_i, \Delta_i - d_i\},
 m_i^* =\min\{d_{i+1}-\Delta_i,\Delta_i - d_i\}.
 \end{split}
\end{equation}
This provides sufficient conditions for the convexity of the
classical rational cubic spline [ \cite{DG4}, p. 971].
\end{rem}
\begin{rem}\label{rema}
In particular,  if we choose
\begin{equation}\label{38}
 r_i =  1 +\frac{M_i}{m_i}+\frac{m_i}{M_i},
\end{equation}
 with $\alpha_i$ satisfying $0 \le \alpha_i < \min \Big \{  a_i ^2, \frac{h_i (d_{i+1} - \Delta_i)}{d_N (x_N - x_1) - (y_N - y_1)},    \frac{h_i ( \Delta_i - d_i )}{ (y_N - y_1) -d_1 (x_N - x_1) } \Big \}$
  so as to settle the convexity in question, then the rational cubic spline FIF $S$ in
(\ref{r10}) reduces to a lower-order form given by
 \begin{equation}\label{r38a}
S\big(L_i(x)\big) = \alpha_i S(x)  + (1-\theta)(y_i-\alpha_i y_1)+
\theta (y_{i+1}-\alpha_i y_N)-\dfrac{h_i \theta(1-\theta) G_i
H_i}{G_i (1-\theta)+ H_i\theta},
\end{equation}
\begin{equation*}
\begin{split}
\text{where}\; G_i &:=  (d_{i+1} - \Delta_i ) - \frac{ \alpha_i}{h_i} [d_N(x_N-x_1) - (y_N - y_1)]\quad \text{and} \\
H_i&:=(\Delta_i-d_i)-\frac{\alpha_i}{h_i}[(y_N-y_1)-d_1(x_N-x_1)].\\
\end{split}
\end{equation*}
The classical counterpart of (\ref{r38a}) obtained by choosing all
the scaling factors to be zero is described in \cite{D2}. In other
words, (\ref{r38a}) yields a fractal generalization of the
classical rational spline with quadratic numerator and linear
denominator
studied in \cite{D2}.\\
Our particular choice of shape parameters  given in (\ref{38})
verifying the convexity condition can be justified as follows. For
the shape parameters as in (\ref{38}), and the scaling factors
satisfying $|\alpha_i| < a_i ^4$, we have $r_i-3=O(h_i^2)$, and
consequently we obtain optimal $O(h^4)$ bound on interpolation
error provided derivatives are estimated with $O(h_i^3)$ accuracy.
\end{rem}
\noindent The main points in the above discussion are extracted in
the form of following theorem:
\begin{thm}\label{C1}
Given a convex (concave) data $\{(x_i, y_i) : i = 1,2,\dots,N\}$,
assume that the derivative parameters satisfy the necessary
convexity (concavity) condition. Then, the following conditions on
the scaling factors and the shape parameters are sufficient for
the corresponding $\mathcal{C}^1$-rational cubic spline FIF $S$ to
be convex (concave) on $I= [x_1, x_N]$. \small
\begin{equation*}
\begin{split}
 & 0 \le \alpha_i   \le \min \Big \{\kappa a_i ^2,\kappa^* \frac{h_i (d_{i+1}
- \Delta_i)}{d_N (x_N - x_1) - (y_N - y_1)}, \kappa_* \frac{h_i (
\Delta_i - d_i )}{ (y_N - y_1) -d_1 (x_N - x_1) } \Big \};
\kappa,\kappa^*, \kappa_* \in [0,1),\\
&   r_i \ge  \max \Bigg  \{ \frac{d_{i+1}- d_i +
(\alpha_i/h_i)(x_N - x_1)  (d_1 - d_N)} { d_{i+1} - \Delta_i -
(\alpha_i/h_i) [ d_N (x_N - x_1) - (y_N - y_1)]},
      \frac {d_{i+1}- d_i + (\alpha_i/h_i)   (x_N - x_1)  (d_1 - d_N)} {\Delta_i - d_i -(\alpha_i/h_i) [y_N - y_1 - d_1 (x_N - x_1)]}\Bigg \}.
\end{split}
\end{equation*}
\normalsize
\end{thm}
\subsection{Convex and Monotonic Data}\label{r1ps43}
We now consider the possibility that the data satisfy both the
monotonic increasing condition $ y_1 < y_2< \dots <y_N$, and  the
strictly convex condition (\ref{cc}). The derivative parameters
must then satisfy the following inequalities:
\begin{equation}\label{r39}
 0 \le d_1 < \Delta_1 < d_2 < \Delta_2 < \dots < \Delta_{i-1}< d_i < \Delta_i < \dots < d_N.
\end{equation}
\noindent We claim that the convex interpolation method described
in the previous subsection is suitable for obtaining a convex and
monotonic fractal interpolant. To verify this claim, we proceed as
follows. Assume that the sufficient conditions (\ref{r36}) on the
scaling factors that achieve the convexity  of the rational cubic
FIF hold. Rearrangement  of these inequalities gives
$$ 0 \le  \alpha_i < \frac{h_i (d_{i+1} - \Delta_i)}{d_N (x_N - x_1) - (y_N - y_1)}
\implies  \Delta_i  - \frac{\alpha_i}{h_i} (y_N - y_1) < d_{i+1} -
\frac{\alpha_i}{h_i} d_N (x_N - x_1)$$ and
$$ 0 \le  \alpha_i <  \frac{h_i ( \Delta_i - d_i )}{ (y_N - y_1) -d_1 (x_N - x_1)}
\implies  d_i -  \frac{\alpha_i}{h_i} d_1 (x_N - x_1) < \Delta_i -
\frac{\alpha_i}{h_i} (y_N - y_1).$$ Combining  two inequalities
obtained above, we have
\begin{equation}\label{r40}
 d_i -  \frac{\alpha_i}{h_i} d_1 (x_N - x_1) < \Delta_i  - \frac{\alpha_i}{h_i} (y_N - y_1) < d_{i+1} -  \frac{\alpha_i}{h_i} d_N (x_N - x_1)
\end{equation}
Since $a_i < 1 $, the condition  $\alpha_i  \le  \kappa a_i^2$
given in (\ref{r36}) implies the condition
 $\alpha_i \le \kappa a_i$ in  (\ref{r30}). Also
\begin{equation}\label{r40a}
\alpha_i \le \kappa a_i, d_i \ge 0 ~\forall~ i \in J
\Longrightarrow \ \alpha_i d_1 \le \kappa \frac{h_i}{x_N-x_1}d_1
\Longrightarrow d_i -\frac{\alpha_i}{h_i} d_1 (x_N-x_1) \ge d_i
-\kappa d_1
\end{equation}
Hence from (\ref{r39}) and (\ref{r40a}), we have  $ d_i -
\frac{\alpha_i}{h_i} d_1 (x_N - x_1) \ge  0$. Consequently,
(\ref{r40}) yield $ \Delta_i - \frac{\alpha_i}{h_i} (y_N - y_1)
\ge 0 $ and $ d_{i+1} - \frac{\alpha_i}{h_i} d_N (x_N - x_1) \ge 0
$. Thus, we get the sufficient condition on the scale
factors $\alpha_i (i = 1,2,\dots, N-1)$ that retain the data monotonicity. \\
Assume that the sufficient condition (\ref{r37})  on the shape
parameters $r_i$  for  the convexity of the  rational cubic FIF is
true. We will prove that, this condition implies the condition
(\ref{r27}) on $r_i$ for the monotonicity of the rational cubic
fractal interpolant. Without loss of generality, assume that
$$ M_i = [d_{i+1} - \frac{\alpha_i}{h_i} d_N (x_N - x_1) ]- [\Delta_i - \frac{\alpha_i}{h_i} (y_N - y_1)],$$
$$m_i = [\Delta_i - \frac{\alpha_i}{h_i} (y_N - y_1)] - [d_i - \frac{\alpha_i}{h_i} d_1 (x_N - x_1) ].$$
Denote $P_i^* = d_i - \frac{\alpha_i}{h_i} d_1 (x_N - x_1), Q_i^*
= \Delta_i - \frac{\alpha_i}{h_i} (y_N - y_1), R_i^* = d_{i+1} -
\frac{\alpha_i}{h_i} d_N (x_N - x_1)$.  From (\ref{r40}), we have
$ P_i^* \le Q_i^* \le R_i^* $. Again, with these notations
\begin{equation}\label{r42}
 M_i = R_i^* - Q_i^* = \max \{ R_i^* - Q_i^*,  Q_i^* - P_i^* \} \implies Q_i^* \le \frac{P_i^*+R_i^*}{2}.
\end{equation}
The sufficient condition (\ref{r27}) for the monotonicity of a
rational cubic FIF can be rearranged as
\begin{equation}\label{r43}
r_i \ge \frac{d_i + d_{i+1} - \frac{\alpha_i}{h_i}(d_1+d_N)(x_N -
x_1)}{\Delta_i - \frac{\alpha_i}{h_i} (y_N - y_1)} = \frac{P_i^* +
R_i^*}{Q_i^*}.
\end{equation}
The sufficient condition  (\ref{r37}) for the convexity of a
rational cubic FIF  in above notations becomes
\begin{equation}\label{r44}
r_i \ge 1 + \frac{R_i^* - Q_i^*}{Q_i^* - P_i^*} = \frac{R_i^* -
P_i^*}{Q_i^* - P_i^*}.
\end{equation}
Note that (\ref{r44}) implies (\ref{r43}), if $ \frac{R_i^* -
P_i^*}{Q_i^* - P_i^*} \ge \frac{P_i^* + R_i^*}{Q_i^*}$ which is
equivalent to the condition described in (\ref{r42}). But the
condition (\ref{r42}) is obviously true due to our assumptions.
The proof is similar if we assume that $M_i = Q_i^* - P_i^* $ and
$m_i = R_i^* - Q_i^*$. Thus, we have proved the sufficient
condition for the convexity of a rational cubic FIF on shape
parameters $r_i$ gives the sufficient condition on  for the
monotonicity $r_i$. Therefore we conclude that for a given
monotonic increasing convex data set, if derivative parameters are
chosen according to (\ref{r39}), then convex interpolation scheme
developed in Section 4.2 will automatically produce a convex
monotone rational cubic spline FIF.
\begin{thm}\label{CM1}
Given a set of strictly  convex monotonic increasing data $\{(x_i,
y_i) : i = 1,2,\dots,N\}$, assume that the derivative parameters
satisfy the necessary condition expressed in (\ref{r39}). Then a
convex interpolant obtained through the convexity preserving
rational FIF scheme in  Theorem (\ref{C1}) will automatically
render a convex and monotone fractal interpolation curve.
\end{thm}
\section{Some Remarks and Possible Extensions}\label{r1ps5}
\begin{enumerate}[(i)]
\item \emph{Preserving Positivity}: The proposed rational FIF can
also generate positive fractal curves for a given set of positive
data $\{(x_i,y_i):i=1,2,\dots, N\}$. Recall that the rational
cubic spline FIF $S$ is generated iteratively using the functional
equation $S\big(L_i(.)\big)=\alpha_i S(.) +R_i(.)$. Hence with
$\alpha_i \ge 0$ for all $i \in J$, the conditions $R_i(x) \ge 0$
for all $x \in I$ and for all $i \in J$ is enough to ensure $S(x)
\ge 0$ for all $x \in I$. Since $R_i$ has positive denominator,
the positivity of $R_i(x)$ reduces to the positivity of cubic
polynomial $P_i(x)$. Computationally efficient sufficient
conditions for the positivity of $P_i(x)=P_i^*(\theta)$ is given
by $A_i \ge 0$, $B_i \ge 0$, $C_i \ge 0$, and $D_i \ge 0$. With
simple calculations we obtain the following conditions on the IFS
parameters: $0 \le \alpha_i < \min \Big \{ a_i, \frac {y_i }{y_1},
\frac {y_{i+1}} { y_N } \Big \}$ and $r_i> \max\{-1,
\frac{-h_id_i+\alpha_i
d_1(x_N-x_1)}{y_i-\alpha_iy_1},\frac{h_id_{i+1}-\alpha_i
d_N(x_N-x_1)}{y_{i+1}-\alpha_iy_N}\}$. In particular, for all
$\alpha_i=0$, we obtain conditions for the positivity of the
rational cubic spline introduced in \cite{DG4}. It seems that
\cite{DG4} does not address the possibility of preserving
positivity with the rational cubic spline developed therein.
 \item \emph{Admissibility of negative
scalings for shape preserving}: By taking monotonicity as an
example of shape, we illustrate that the nonnegativity assumption
on the scaling parameters is not essential in our shape preserving
schemes. For this purpose, we outline a slightly general problem,
namely, identifying  the parameters of the IFS so that the graph
of $S^{(1)}$ lies in a prescribed rectangle $R=I \times [0,M]$.
Recall that $S^{(1)}$ is generated using the IFS
$\big\{\big(L_i(x), F_{i,1}(x,y)\big)\big\}$ where $F_{i,1}(x,y)=
\frac{\alpha_i y + R_i^{(1)}(x)}{a_i}.$ By the properties of the
attractor of the IFS, for $0 \le S^{(1)}(x)\le M$ it suffices to
prove that $F_{i,1}(x,y) \in [0, M]$ for any $(x,y) \in R$.
Consider the two cases: (i) $0 \le \alpha_i < a_i$ (ii) $ -a_i<
\alpha_i <0$. With $0 \le \alpha_i < a_i$, $0 \le y \le M$, the
condition $F_{i,1}(x,y) \in [0, M]$ holds if $0 \le R_i^{(1)}(x)
\le M(a_i-\alpha_i)$. Now by the substitution of the rational
expression $R_i^{(1)}(x)$, the above inequality can be transformed
to the positivity of suitable quartic polynomials. This provides
conditions on IFS parameters for case (i). For case (ii), the
condition  $F_{i,1}(x,y) \in [0, M]$  will hold if $ -\alpha_i M
\le R_i^{(1)}(x) \le a_i M$. Again, this can be transformed to the
positivity of suitable quartic polynomials. From this the
conditions on the IFS parameters for case (ii) are deduced.
Combining the conditions in both cases, we obtain sufficient
conditions on the IFS parameters for the graph of $S^{(1)}$ to lie
in $R$. Taking $M$ to be a large positive number, we deduce
conditions for the monotonicity of $S$. Note that this allows
negative values for the scaling parameters. \item
\emph{Co-monotone/co-convex fractal interpolants}: Often a data
set will not be globally monotone, but instead switches back and
forth between monotone increasing and monotone decreasing. We need
the interpolant to follow  the shape of the data in the following
sense: $S$ is monotonically increasing/decreasing on $[x_r, x_s]$
if the data are monotonically increasing/decreasing on $[x_r,
x_s]$. Similarly we can define co-convex interpolation problem. If
co-monotone/co-convex interpolation with the present fractal
scheme is desired, a preliminary subdivision of the points into
subsets of uniform shape is needed. Let us illustrate this with an
example. Consider a data set $\{(x_i, y_i) : i=1,2,\dots, 7\}$
where $y_1 \le y_2 \le y_3$;\; $y_3 \ge y_4 \ge y_5 \ge y_6$; \;
$y_6 \le y_7$. In order to achieve co-monotonicity, we must insure
that slopes at transition points are zero, i.e., $d_3=0$ and
$d_6=0$. We divide the interval $I=[x_1,x_7]$ into three
subintervals such that the data possess same type of monotonicity
property throughout that subinterval; $I_1=[x_1,x_3]$,
$I_2=[x_3,x_6]$, $I_3=[x_6,x_7]$. We can apply the developed
monotonicity preserving  FIF algorithm to obtain a monotonically
increasing rational cubic spline FIF $S_1$ on $I_1$. With proper
renaming of the  data points if necessary, the monotonically
decreasing FIF algorithm can be applied to obtain a rational cubic
spline FIF $S_2$ on $I_2$. Now consider the interval
$I_3=[x_6,x_7]$ which contains only two data points. Here the
iterations of IFS code cannot produce any new points. To remedy
this problem, we introduce a new node say, $(x_6^*,y_6^*)$ that is
consistent with the shape present in $[x_6,x_7]$, i.e., $x_6
<x_6^*<x_7$ and $y_6 \le y_6^* \le y_7$. The ``best'' choice of
additional node  deserves further research. We apply the developed
monotonically increasing FIF scheme with an arbitrary but shape
consistent extra node to obtain a cubic spline FIF $S_3$, which is
co-monotone with the data in $I_3$. Construct a rational cubic
spline FIF $S$ in a piecewise manner by defining $S|_{I_i}=S_i$.
Note that the Hermite interpolation conditions on $S_i$ provide
the $\mathcal{C}^1$-continuity for $S$. \item \emph{Optimal choice
of parameters}: Without a doubt, the scaling parameters and the
shape parameters together provide a large flexibility in the
choice of an interpolant. Consequently, a natural question on an
``optimal'' choice of the parameter arises. In this regard, some
remarks are in order. For higher irregularity in the derivative
(quantified in terms of the fractal dimension) we have to choose
large values of the scaling factors. In contrast, for localness of
the scheme small values of $\alpha_i$ are preferred.  A preferable
choice of the scaling factors and the shape parameters for a
monotonicity/convexity preserving rational cubic spline FIF in
terms of optimal interpolation error is given in Remarks
(\ref{remb}),(\ref{rema}). Among the various shape preserving
rational fractal interpolants a visually improved solution may be
obtained by minimization of a fairness functional such as Holladay
functional. The feasible domain is given by suitable restriction
on the IFS parameters. This results in a constrained nonlinear
optimization problem. In the classical non-recursive shape
preserving schemes, widely used Holladay functional is
$\int\limits_{ x_1}^{x_N}\big[ S^{(2)}(x)\big]^2 \,\mathrm{d}x.$
However, for the present fractal scheme, $S^{(2)}$ may have
discontinuity at each point of the interval $I=[x_1,x_N]$, and
consequently computing the integral occurring in this Holladay
functional would be impossible at least in the Riemann sense. If
we interpret the second derivative occurring in the Holladay
functional loosely as $S^{(2)}(x^+)$ or assume that the FIF is in
$\mathcal{C}^2$, then
\begin{equation*}
\begin{split}
M&=\int\limits_{I}\big[ S^{(2)}(x)\big]^2 \,\mathrm{d}x =\underset
{i \in J}\sum \int\limits_{I_i}\Big[ \frac{\alpha_i}{a_i^2}
S^{(2)}(L_i^{-1}(x)) +
\frac{1}{a_i^2}R_i^{(2)}(L_i^{-1}(x))\Big]\,\mathrm{d}x,\\
&= \Big[\underset{i \in J}\sum
\frac{\alpha_i}{a_i^2}\int\limits_{I_i}S^{(2)}(L_i^{-1}(x))\,\mathrm{d}x\Big]+
R_0,
\end{split}
\end{equation*}
where $R_0=\int\limits_{I}R^*(x)\,\mathrm{d}x$ and
$R^*(x)=\frac{1}{a_i^2}R_i^{(2)}(L_i^{-1}(x))$, if $x \in I_i$.
Applying the change of variable $\tilde{x}=L_i^{-1}(x)$,
\begin{equation*}
M = \underset {i \in J}\sum \frac{\alpha_i}{a_i} M + R_0
\Longrightarrow M = \frac{R_0}{1-\underset {i \in J}\sum
\frac{\alpha_i}{a_i}}.
\end{equation*}
Thus, the constrained optimization problem is to Minimize $M$
where variables are restricted according to finite set of
inequalities resulting from the shape preserving constraints. It
is felt that this constrained optimization problem can be solved
by means of a differential evolution optimization
algorithm/genetic algorithm. This procedure is justified if we
make the interpolant to be $\mathcal{C}^2$ by imposing suitable
conditions on the derivative values and the IFS parameters
resulting from the $\mathcal{C}^2$-continuity conditions. This may
be done on lines similar to the cubic spline FIFs (see \cite {CK},
\cite{CV}). This will lead to the fractal generalization of
the standard $\mathcal{C}^2$-rational cubic spline introduced by Gregory \cite{G}. \\
From the point of view of approximation theory, the problem of
finding optimal rational spline FIF $S$ is an inverse problem
which reads as: Given a set of values of a function, recover the
IFS parameters generating this target function. Levkovich
\cite{LM} has obtained contraction affine mappings generating a
given function based on the connection between the maxima skeleton
of wavelet transform of the function and positions of the fixed
points of the affine mappings in question. It is not without
interest to note that for adapting a similar technique, the
connection between the strongest singularities of the FIF or its
derivative and fixed points of the generating mappings, which is
non- affine in the present case, is to be developed rigourously.
Lutton et al. \cite{LL} have applied genetic algorithms for
solving this type of inverse problems. Resolution of the inverse
problem is a major challenge of both theoretical and practical
interest, which is not settled in its full generality.
\item{\emph{A comparison of the present method with subdivision
methods}:} A possible alternative to the present recursive shape
preserving interpolant schemes that introduce fractality in the
derivatives is so-called shape preserving subdivision schemes
(see, for instance, \cite{CC, DL, TM}). Now we briefly compare and
contrast the two methodologies. In both these interpolation
methods, the desired interpolant is obtained constructively.
Convergence of the fractal interpolation scheme and the
differentiability of the limit function follow from a straight
forward application of the Banach contraction principle on a
suitable function space. Establishing the convergence of the
scheme and the differentiability of the limit function is
relatively harder in the subdivision schemes. Though the
subdivision schemes add fractality to the derivative function,  we
cannot directly control this fractality in terms of the parameters
involved in the scheme. On the other hand, it is known \cite{B}
that as magnitudes of the scaling factors are increased from zero,
the dimension of the derivative of the fractal spline increases.
By controlling the scaling factors, the fractality can be
considered in a small portion of the domain, if in this part
possible signal displays some complex disturbance. A quantitative
measure of the irregularity (fractality), namely, box
counting/Hausdorff dimension of the fractal curves in terms of the
scaling parameters involved in the IFS is obtained in
\cite{DD2,B}. Up to our knowledge, a quantitative measure of the
irregularity of the derivative in terms of  the parameters
involved is unavailable in subdivision schemes. Using the notions
of hidden variable FIFs and coalescence hidden variable FIFs the
present scheme can be extended to preserve shape of the data
generated from a self-affine, non-self-affine, or partly
self-affine and partly non-self-affine function. However,
subdivision schemes do not specify about these properties of the
constructed interpolants. The main appeal to the subdivision
schemes resides in their localness. Due to the recursive and
implicit nature of the FIF, the proposed scheme is, in general,
non-local. However, the completely local classical non-recursive
interpolation scheme emerges as a special case of the proposed
fractal interpolation method, and consequently, our method  is
local or global depending on the magnitude of the scaling factor
in each subinterval.
\end{enumerate}
\section{Numerical Examples}\label{r1ps6}
Iterating the functional equation (cf. (\ref{r10}))  with suitable
choices of the scaling factors and the shape parameters as
prescribed in Section \ref{r1ps4}, we generate different shape
preserving rational cubic spline FIFs in this section.\\ If the
derivatives $d_i$ ($1,2,\dots,N$) are not supplied, estimates of
derivatives are necessary. Methods that associate derivatives with
data points involve estimates based on nearby slopes or data
differences. Depending on the applications, various schemes based
on linear combination (e.g., arithmetic mean method) or
multiplicative combination (e.g., geometric mean method) of
chord-slopes are developed in the literature (see, for instance,
\cite{BFK,DG5}). With the notation $ \Delta_i = \frac{y_{i+1} -
y_i}{h_i}, i \in J$, the three point difference approximation for
the arithmetic mean method is given by $ d_i = \frac{h_i
\Delta_{i-1} + h_{i-1} \Delta_i } {h_{i-1}+ h_i}, \quad  i
=2,3,\dots, N-1$ with end conditions $ d_1 = \Big (1 +
\frac{h_1}{h_2} \Big ) \Delta_1 -  \frac{h_1}{h_2} \Delta_{3,1},
\; \Delta_{3,1} = \frac{y_3 - y_1}{x_3 -x_1},$ $ d_N= \Big (1 +
\frac{h_{N-1}}{h_{N-2}} \Big ) \Delta_{N-1} -
\frac{h_{N-1}}{h_{N-2}} \Delta_{N,N-2}, \; \Delta_{N,N-2} =
\frac{y_N - y_{N-2}}{x_N -x_{N-2}}.$
\subsection{Monotonicity Preserving Rational FIFs}
Consider a monotonic data set  $\{(x_i,y_i,d_i)=(0,0,
1.3333),(2,4,2.6666),(3,7,2.6190),(9,9,\\1.5833),(11,13,2.4166)\}$.
For monotonic FIFs, the computed bounds on the scaling factors
are: $0\leq\alpha_1<0.1818,\; 0\leq\alpha_2<0.0985,\;
0\leq\alpha_3<0.1538$,\;
$0\leq\alpha_4<0.1818$.\\
Since $S$ is obtained by iterating the functional equation
(\ref{r10}), perturbation of a particular scaling factor
$\alpha_i$ and/or  shape parameter $r_i$  may ripple through the
entire configuration, i.e., interpolant is potentially non-local.
However, we observed that the portions of the interpolating curve
pertaining to other subintervals are not extremely sensitive
towards changes in the parameters of a particular subinterval. To
illustrate this, we take the monotonic rational cubic spline FIF
in Fig. 1(a) as a reference curve, and analyze the effect of
perturbing the parameters of a particular portion of this curve.
Values of the parameters (rounded off to four decimal places)
corresponding to various curves that are calculated according to
the prescription in Theorem \ref{t2} are given in Table \ref{T1}.
Changing the scaling parameter $\alpha_1$ to $0.05$ (see Table
\ref{T1}), we obtain Fig. 1(b). It is observed that the
perturbation in $\alpha_1$ effects the rational fractal
interpolant considerably in the interval $[x_1, x_2]$, and there
are  no perceptible changes in other subintervals. Similarly, Fig.
1(c) and Fig. 1(d) are obtained by changing the scaling factor
$\alpha_2$ and the shape parameter $r_3$ with respect to the
reference curve. Effects of these changes are observed to be
local. By taking zero scalings in each subinterval, we recapture a
standard rational cubic spline due to Delbourgo and Gregory
\cite{DG4} (see Remarks \ref{rc}, \ref{GDM}) in Fig. 1(e). Optimal
choices of the shape parameters suggested by Remark \ref{remb} are
used to generate
the rational quadratic FIF in Fig. 1(f).  \\
\begin{table}[!h]
\caption{Parameters corresponding to the Monotonic rational cubic
FIFs \label{T1}} \centering
\begin{tabular}{|l |c|  rrrr|}
\hline Figure No &  \multicolumn{5}{c|}{Choice of parameters}
\\ [1ex]
\hline
&$\alpha_i$ & $0.18$ & $0.09$ & $0.1$ & $0.18$ \\[-1ex]
\raisebox{1.5ex}{Fig. 1(a)} &$r_i$ & $2$& $1.8$ & $31$ & $0.5$
\\
\hline
&$\alpha_i$ & $0.05$ & $0.09$ & $0.1$ &$0.18$ \\[-1ex]
\raisebox{1.5ex}{Fig. 1(b)} &$r_i$ & $2$& $1.8$& $31$ &$0.5$
\\
\hline
&$\alpha_i$ & $0.18$ & $0$ & $0.1$ & $0.18$ \\[-1ex]
\raisebox{1.5ex}{Fig. 1(c)}&$r_i$ & $2$& $1.8$ & $31$ & $0.5$
\\
\hline
&$\alpha_i$ &$0.18$& $0.09$& $0.1$ & $0.18$\\[-1ex]
\raisebox{1.5ex}{Fig. 1(d)} & $r_i$
&$2$ & $1.8$ & $350$ & $0.5$  \\
\hline
&$\alpha_i$& $0$& $0$ & $0$ & $0$\\[-1ex]
\raisebox{1.5ex}{Fig. 1(e)}& $r_i$ &$2$ & $1.8$ & $31$ &$0.5$\\[1ex]
\hline
&$\alpha_i$ & $0.001$  & $0$ & $0.08$ &$0.001$ \\[-1ex]
\raisebox{1.5ex}{Fig. 1(f)}& $r_i$ &$2.9961$  & $2.7619$ &
$23.8269$ &$2.9961$\\ \hline
\end{tabular}
\label{tab:PPer}
\end{table}
Let us denote the monotonic rational cubic spline FIFs in Figs.
1(a)-1(f) by $S_i$, $i=1,2,\dots, 6$.  Using the functional
equation (\ref{r22}), the derivative functions  $S_i^{(1)}$
($i=1,2,\dots, 6$) are generated in Figs. 2(a)-2(f). These curves
possess varying irregularity. The derivative $S_5^{(1)}$ of the
classical rational cubic spline is smooth where as $S_1^{(1)}$ is
nowhere differentiable. Note that $S_3^{(1)}$ has smoothness in
the subinterval $[x_2,x_3]=[2,3]$ where the scaling factor is
chosen to be $0$. In this way, the fractality of $S_i^{(1)}$ can
be restricted in a portion of the domain. It can be noted that due
to the small values of the scaling factors in each interval
$S_6^{(1)}$ is almost smooth. The fractal dimension of $S^{(1)}$
constitutes a numerical characterization of the geometry of the
signal and may be used as an index for measuring the complexity of
the underlying phenomenon (see, for instance,\cite{NS2}).
\begin{figure}[h!]
\begin{center}
\begin{minipage}{0.3\textwidth}
\epsfig{file=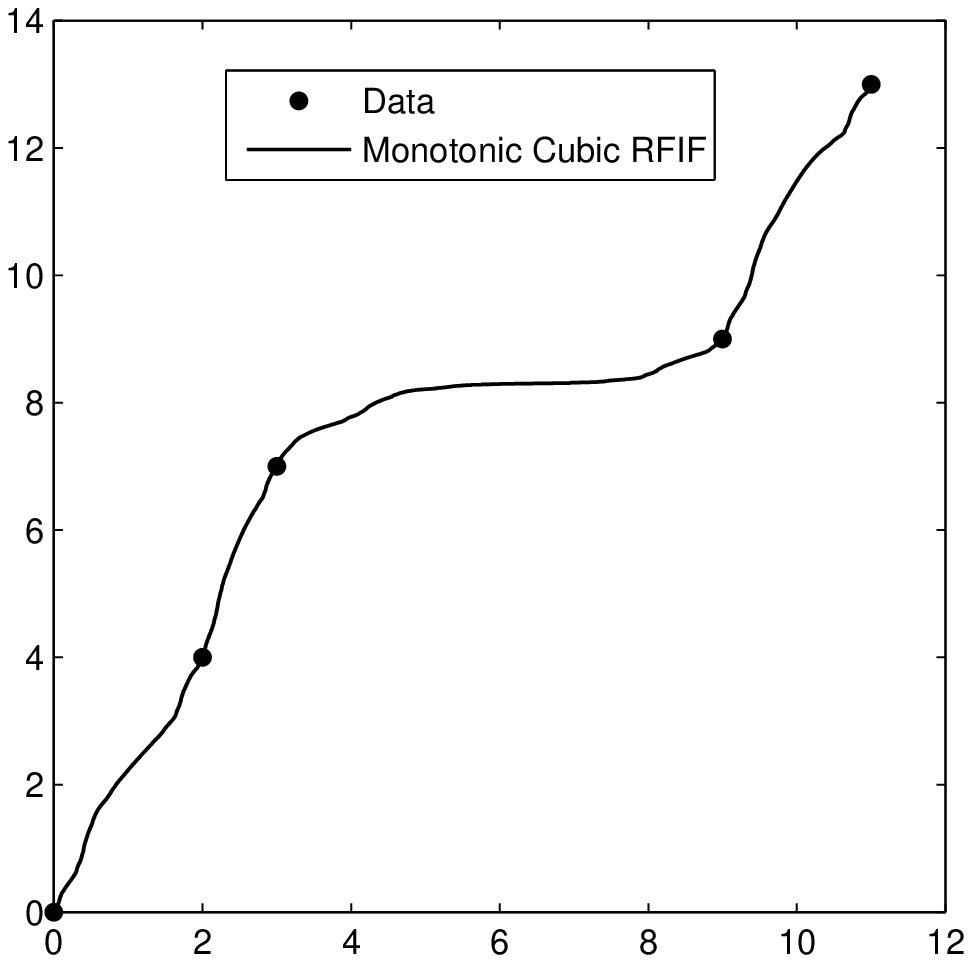,scale=0.38} \centering{\scriptsize{(a):
Monotonic rational cubic FIF $S_1$}}
\end{minipage}\hfill
\begin{minipage}{0.3\textwidth}
\epsfig{file=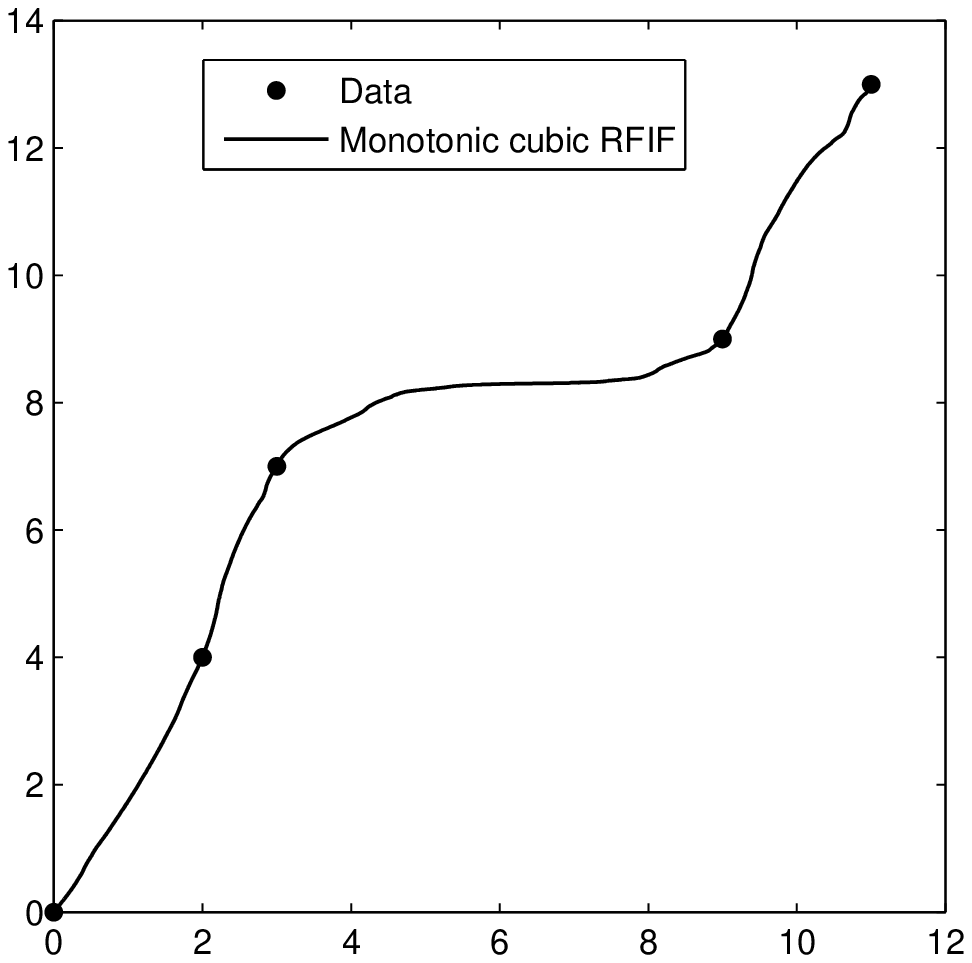,scale=0.38} \centering{\scriptsize{(b):
Monotonic rational cubic FIF $S_2$ (effect of $\alpha_1$)}}
\end{minipage}\hfill
\begin{minipage}{0.3\textwidth}
\epsfig{file = 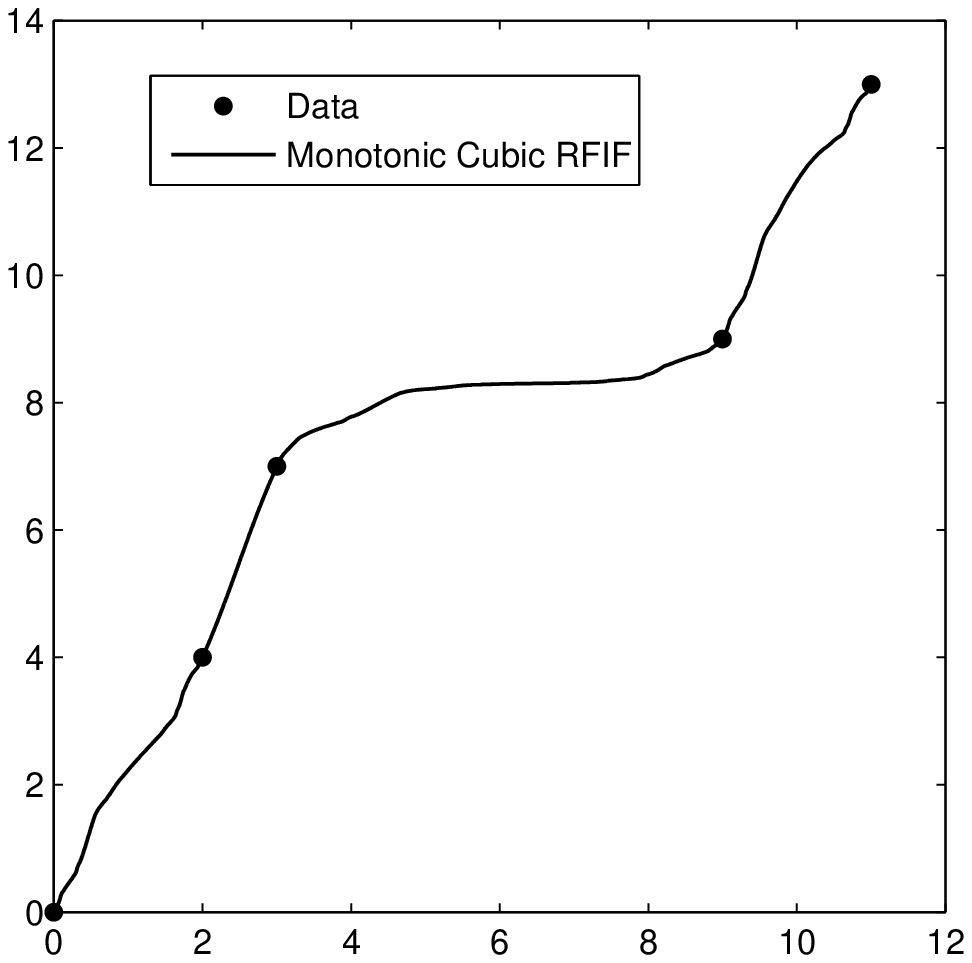,scale=0.38}\\ \centering{\scriptsize{(c):
Monotonic rational cubic FIF $S_3$ (effect of $\alpha_2$)}}
\end{minipage}\hfill
\begin{minipage}{0.3\textwidth}
\epsfig{file=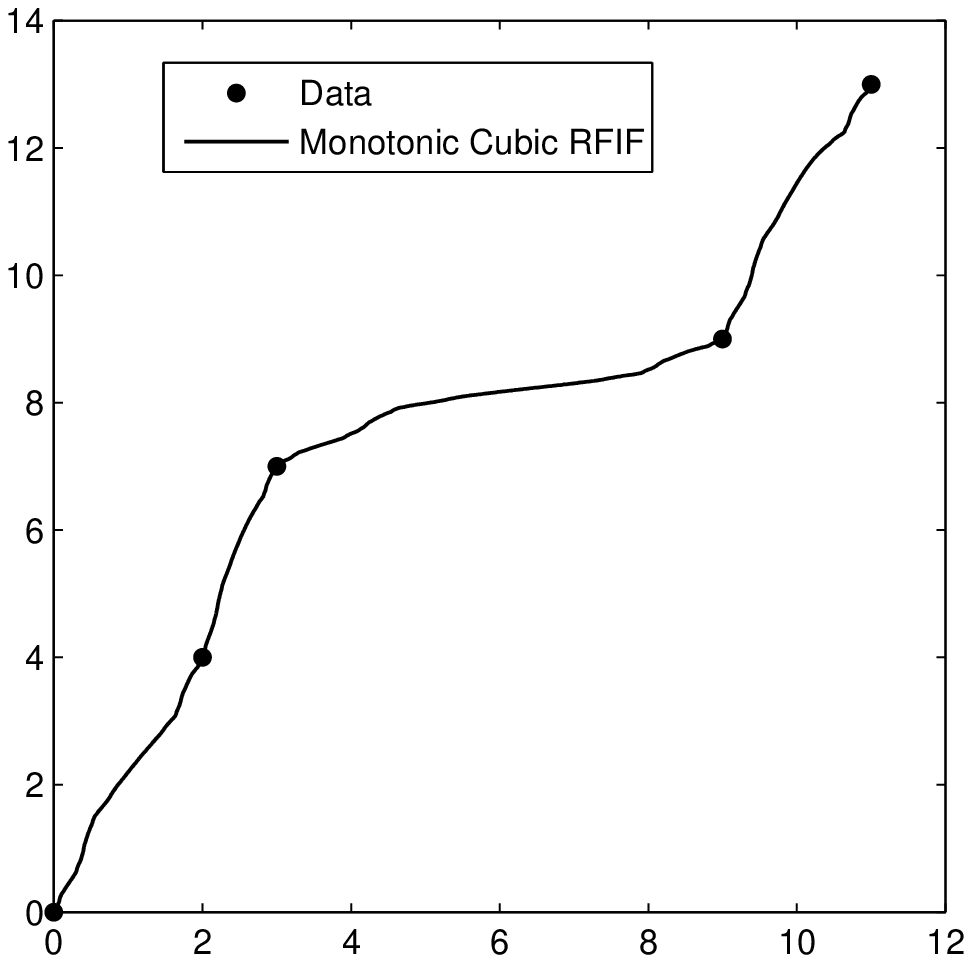,scale=0.38} \centering{\scriptsize{(d):
Monotonic rational cubic FIF $S_4$ (effect of $r_3$)}}
\end{minipage}\hfill
\begin{minipage}{0.3\textwidth}
\epsfig{file=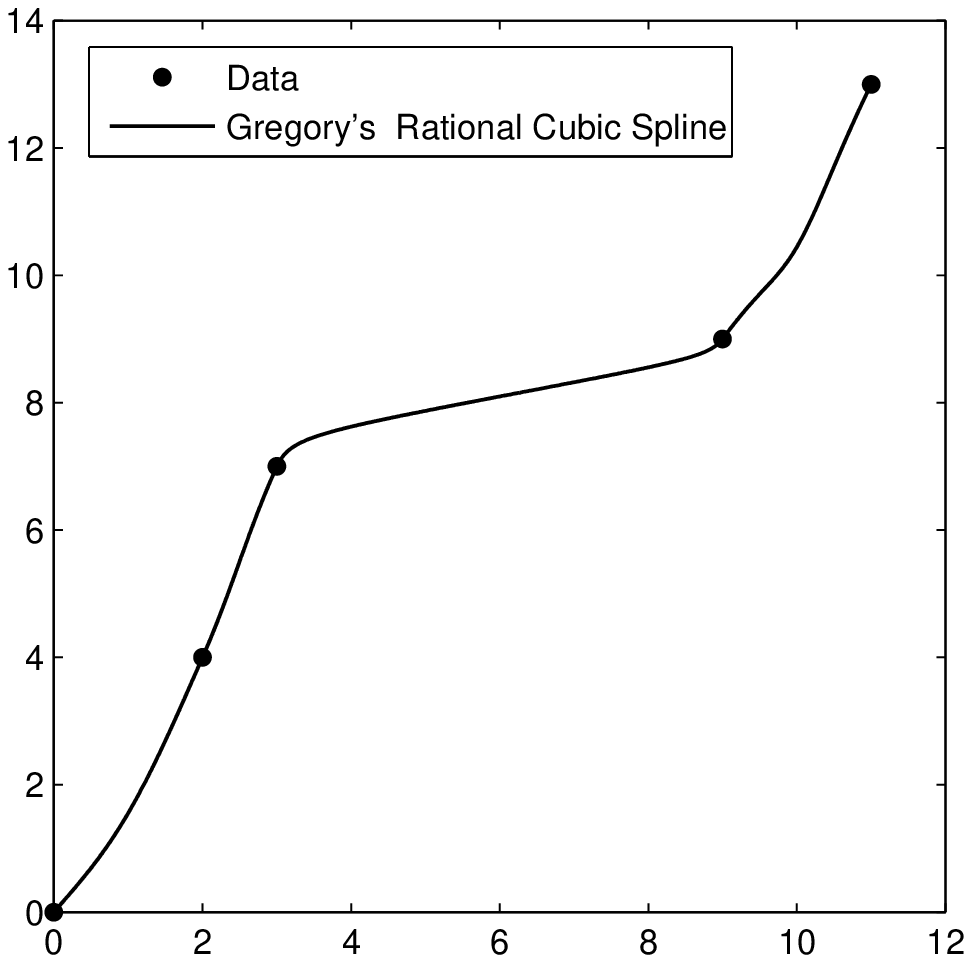,scale=0.38} \centering{\scriptsize{(e):
Classical Monotonic rational cubic spline $S_5$}}
\end{minipage}\hfill
\begin{minipage}{0.3\textwidth}
\epsfig{file=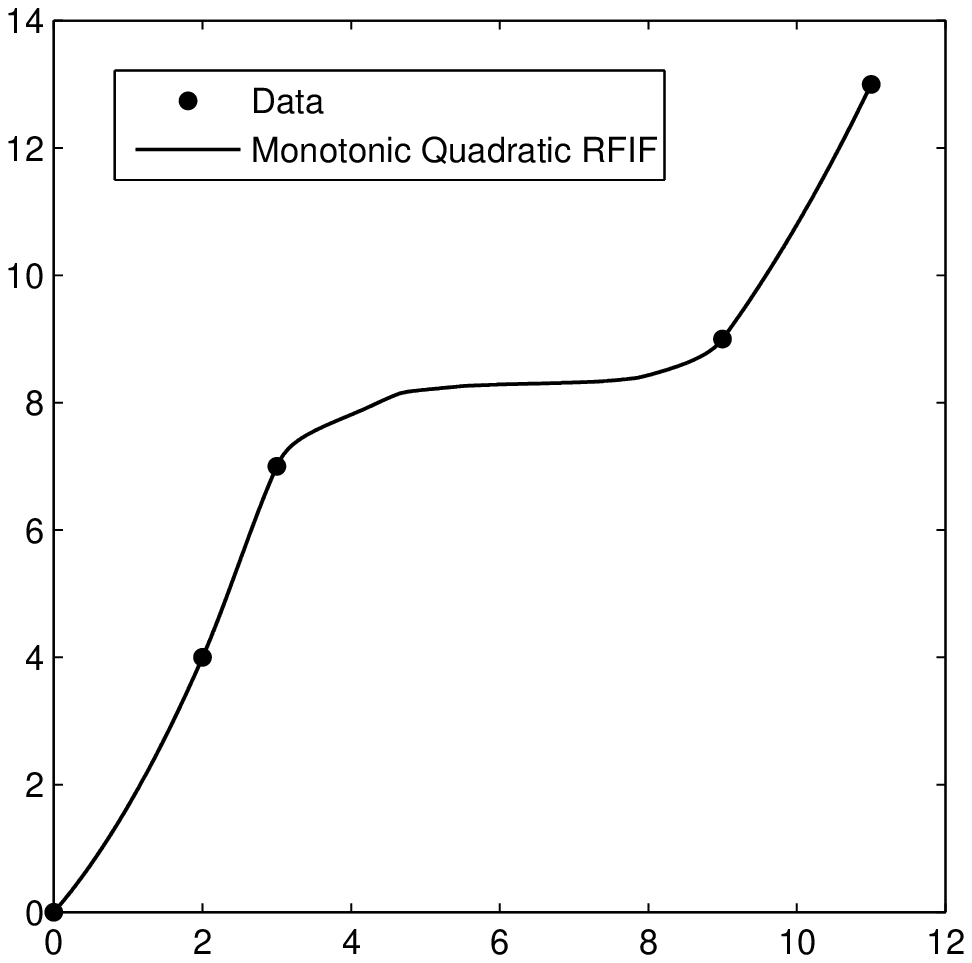,scale=0.38} \centering{\scriptsize{(f):
Monotonic rational quadratic FIF $S_6$}}
\end{minipage}\hfill\\
\caption{Monotonic rational cubic/quadratic spline FIFs: $S_i$,
$i=1,2,\dots,6$.}
\end{center}
\end{figure}
\begin{figure}[h!]
\begin{center}
\begin{minipage}{0.3\textwidth}
\epsfig{file=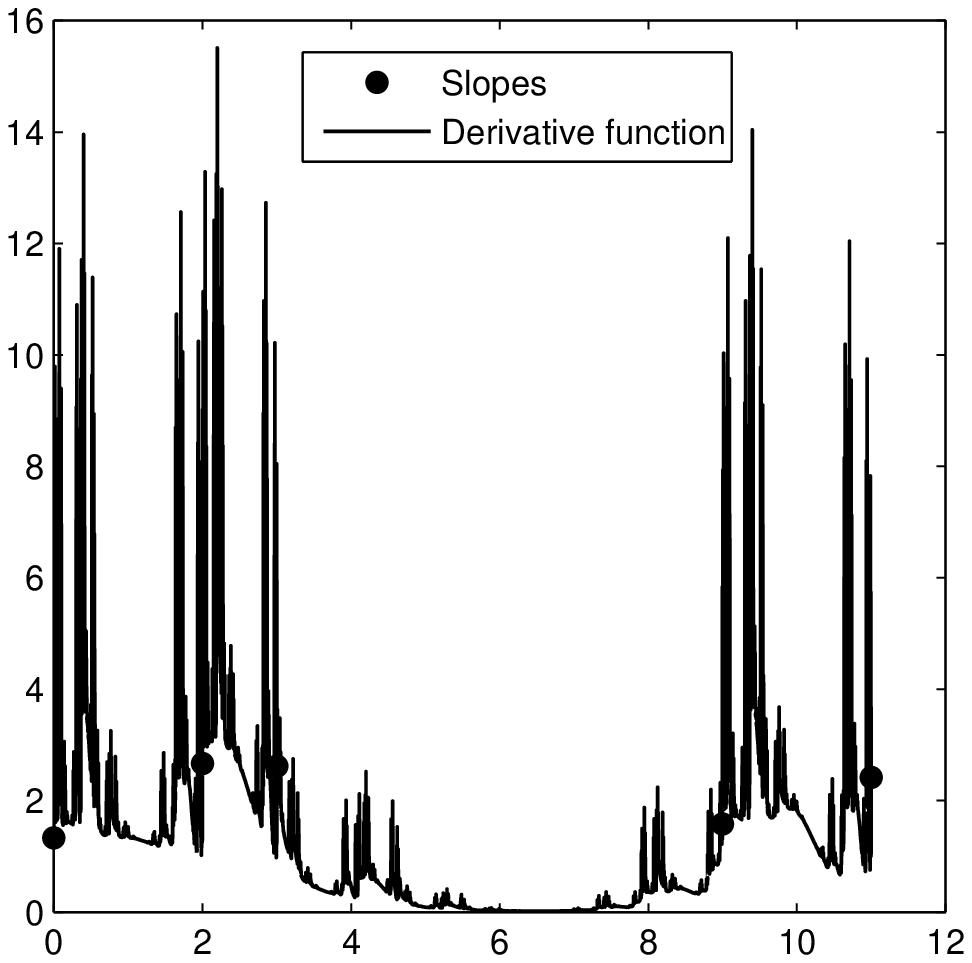,scale=0.38} \centering{\scriptsize{(a):
Fractal function $S_1^{(1)}$}}
\end{minipage}\hfill
\begin{minipage}{0.3\textwidth}
\epsfig{file=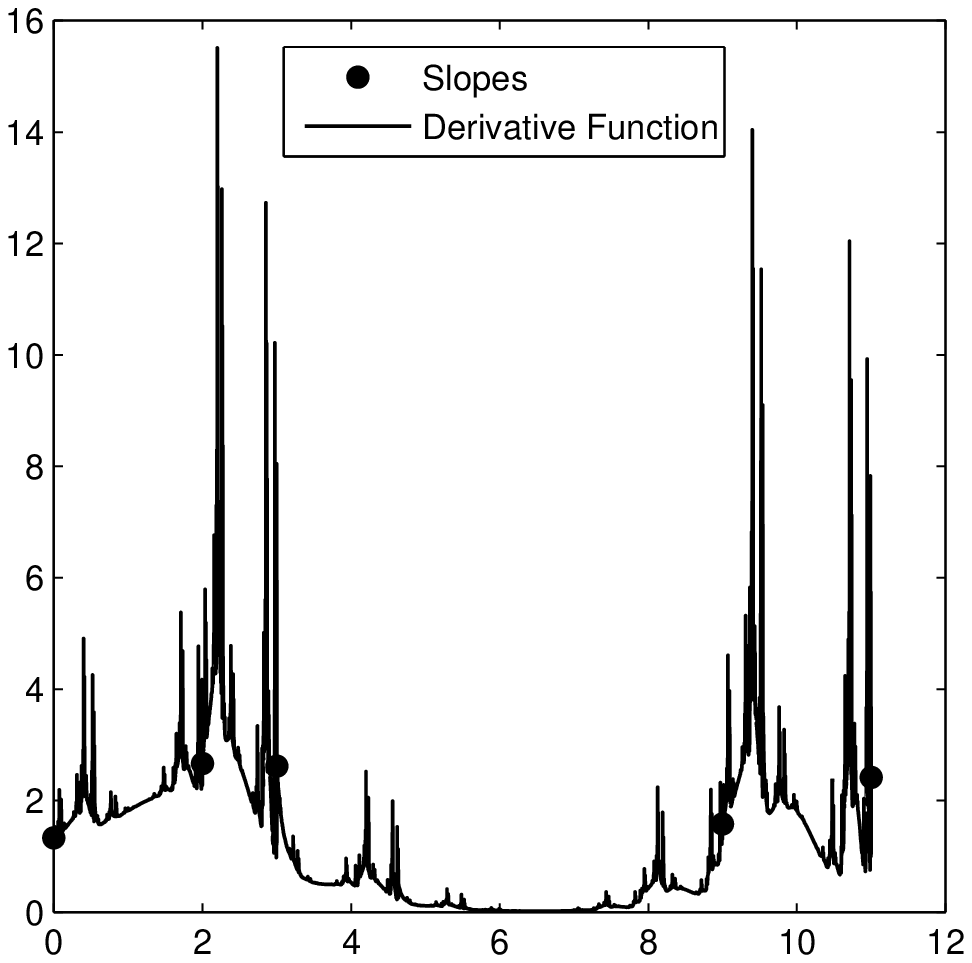,scale=0.38} \centering{\scriptsize{(b):
Fractal function $S_2^{(1)}$}}
\end{minipage}\hfill
\begin{minipage}{0.3\textwidth}
\epsfig{file = 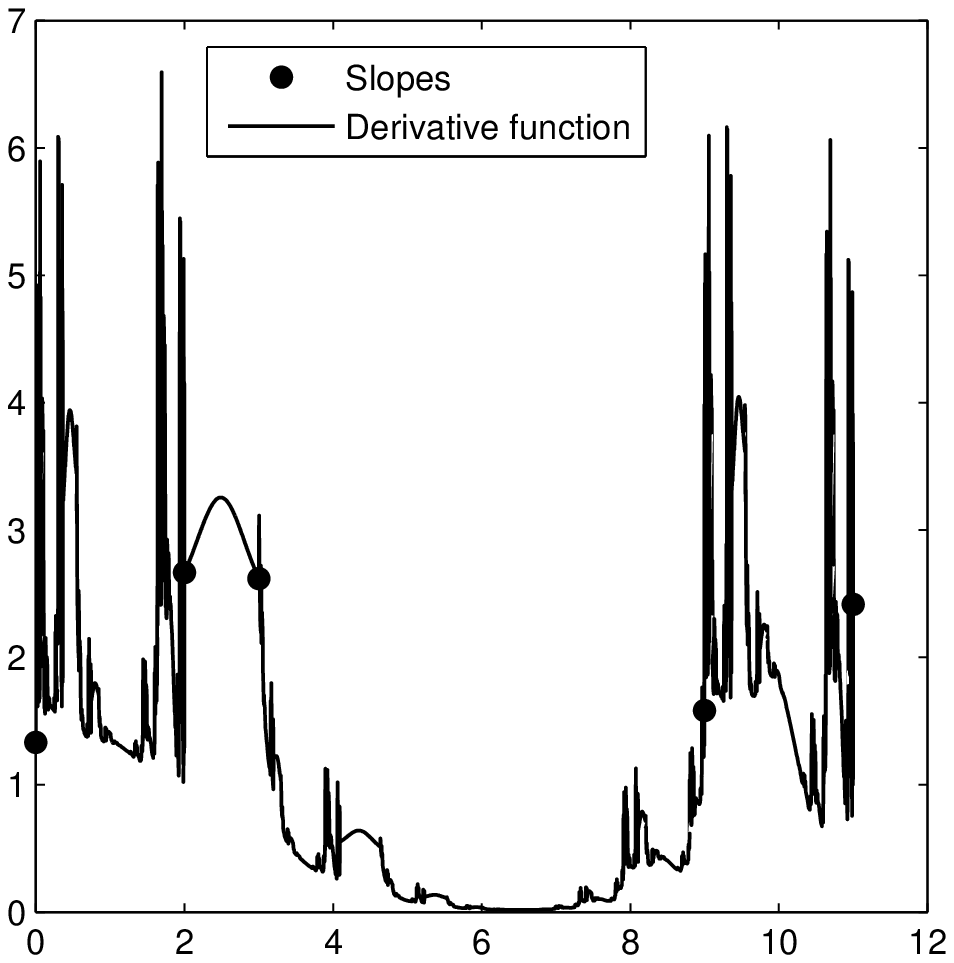,scale=0.38}\\ \centering{\scriptsize{(c):
Fractal function $S_3^{(1)}$ (smooth in $[2,3]$)}}
\end{minipage}\hfill
\begin{minipage}{0.3\textwidth}
\epsfig{file=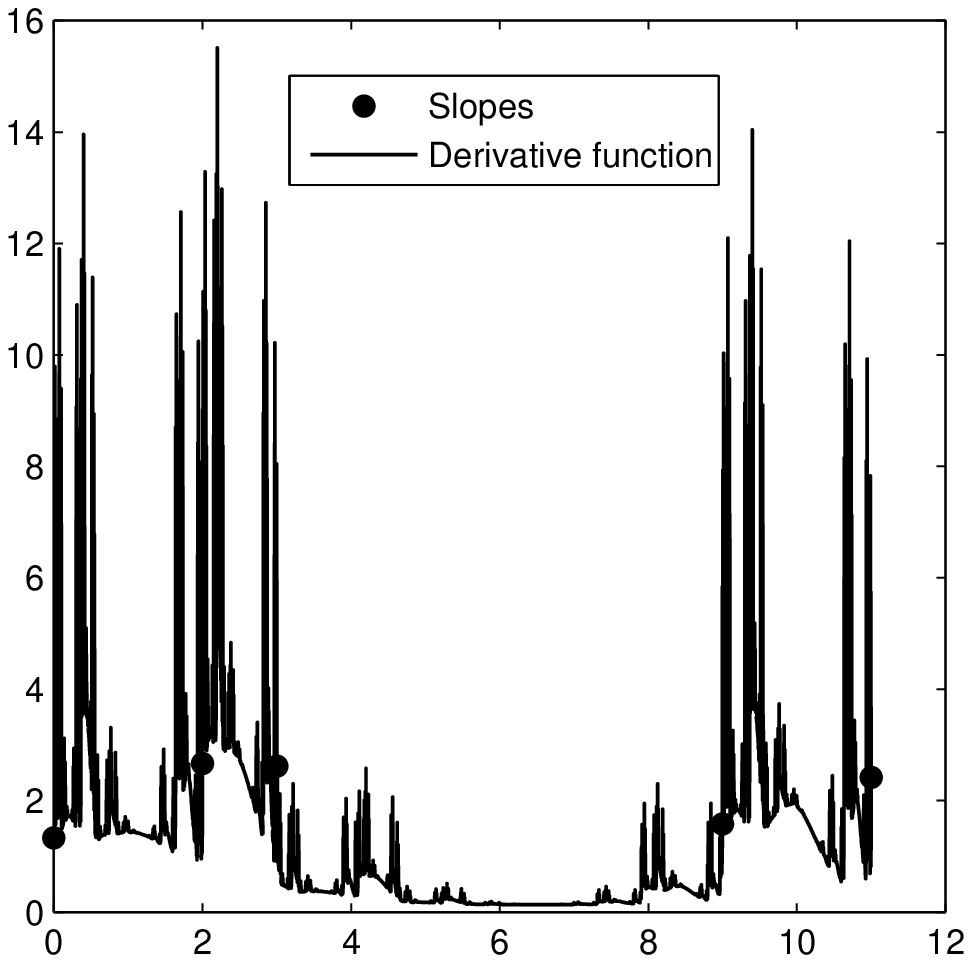,scale=0.38} \centering{\scriptsize{(d):
Fractal function $S_4^{(1)}$}}
\end{minipage}\hfill
\begin{minipage}{0.3\textwidth}
\epsfig{file=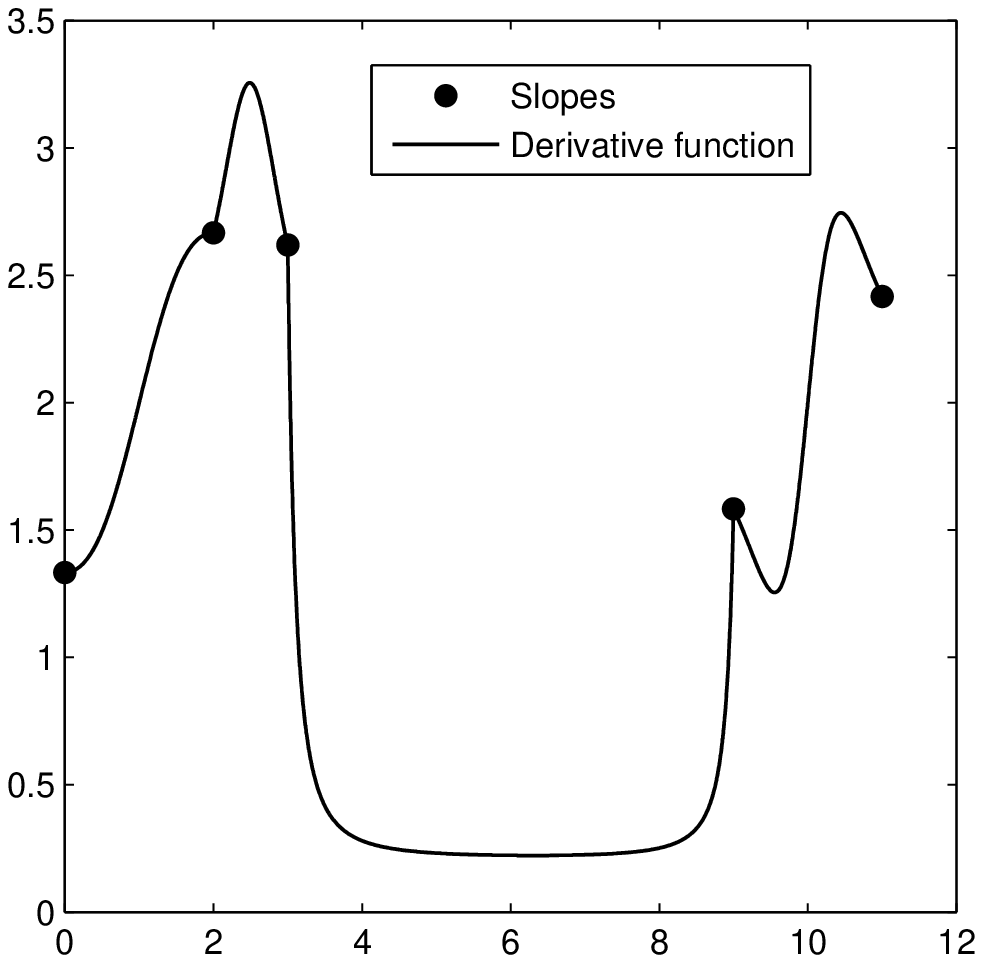,scale=0.38} \centering{\scriptsize{(e):
Function $S_5^{(1)}$}}
\end{minipage}\hfill
\begin{minipage}{0.3\textwidth}
\epsfig{file=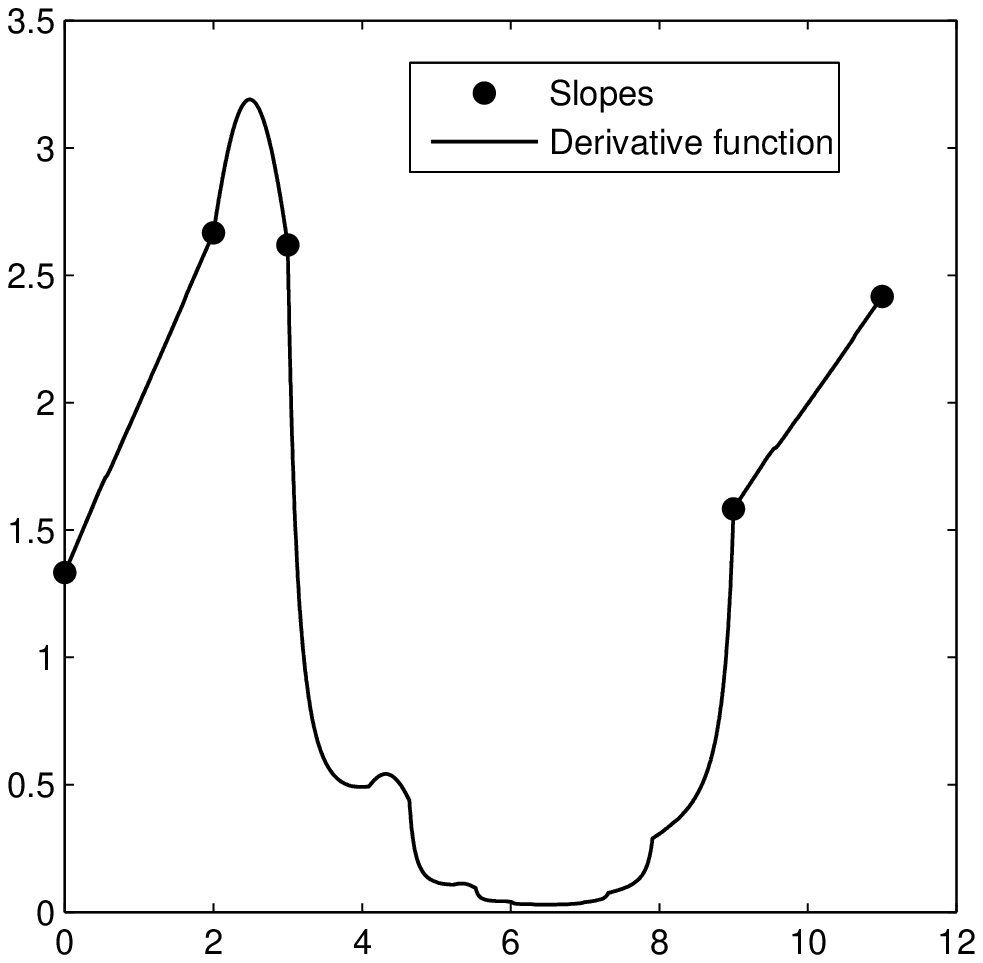,scale=0.38} \centering{\scriptsize{(f):
Function $S_6^{(1)}$}}
\end{minipage}\hfill\\
\caption{Derivatives of monotonic rational FIFs in Figs. 1(a)-1(f)
}
\end{center}
\end{figure}
\subsection{Convexity Preserving Rational FIFs}
Consider the convex data set  $\{(2.2,2), (4,0.625), (5,0.4), (10,
1), (10.22, 1.8)\}$. We computationally generate convex rational
cubic spline FIFs by using (\ref{r10}) and  the parameter values
given by Theorem \ref{C1}. The derivative parameters required for
the implementation of  the IFS scheme are estimated using the
arithmetic mean method. The convex rational cubic spline FIF
generated in Fig. 3(a) is taken as a reference curve.  Changing
the scaling factor $\alpha_3$ to $0.005$ and keeping the values of
the other parameters as in  Fig. 3(a) (see Table \ref{T2}), we
obtain the convex rational cubic spline FIF in Fig. 3(b). It can
be observed that the change in $\alpha_3$ influence the curve only
in $[x_3,x_4]$. Further, due to a small value of the scaling
factor and a large value of the shape parameter the FIF converges
to a line segment in $[x_3,x_4]$, demonstrating the tension
effect. Similar experiments may be conducted by changing the
scalings in other subintervals and the shape parameters. By taking
all the scaling factors to be zero and the shape parameters
according to (\ref{38}), a classical rational quadratic spline
that retain the data convexity is obtained in Fig. 3(c). Thus,
Fig. 3(c) provides a numerical example for the convex rational
quadratic spline by Delbourgo \cite{D2}.  As in the monotonicity
case, it can be observed by plotting the graph of the derivatives
that the scaling factors  provide fractality in the derivatives
$S^{(1)}$ or $S^{(2)}$ (more precisely right hand second
derivative).
\begin{table}[!h]
\caption{Parameters corresponding to the convex rational cubic
FIFs \label{T2}} \centering
\begin{tabular}{|l |c|  rrrr|}
\hline Figure No &  \multicolumn{5}{c|}{Choice of parameters}
\\ [1ex]
\hline
&$\alpha_i$ & $0.02$ & $0.001$ & $0.16$ & $0.007$ \\[-1ex]
\raisebox{1.5ex}{Fig. 3(a)} &$r_i$ & $37$& $8$ & $269$ & $8$
\\[0.5ex]
\hline
&$\alpha_i$ & $0.02$ & $0.001$ & $0.005$ &$0.007$ \\[-1ex]
\raisebox{1.5ex}{Fig. 3(b)} &$r_i$ & $37$& $8$& $269$ &$8$
\\[0.5ex]
\hline
&$\alpha_i$ & $0$ & $0$ & $0$ & $0$ \\[-1ex]
\raisebox{1.5ex}{Fig. 3(c)}&$r_i$ & $3$& $2.6459$ & $12.8069$ &
$3$\\[0.5ex]
 \hline
\end{tabular}
\label{tab:PPer}
\end{table}
\begin{figure}[h!]
\begin{center}
\begin{minipage}{0.3\textwidth}
\epsfig{file=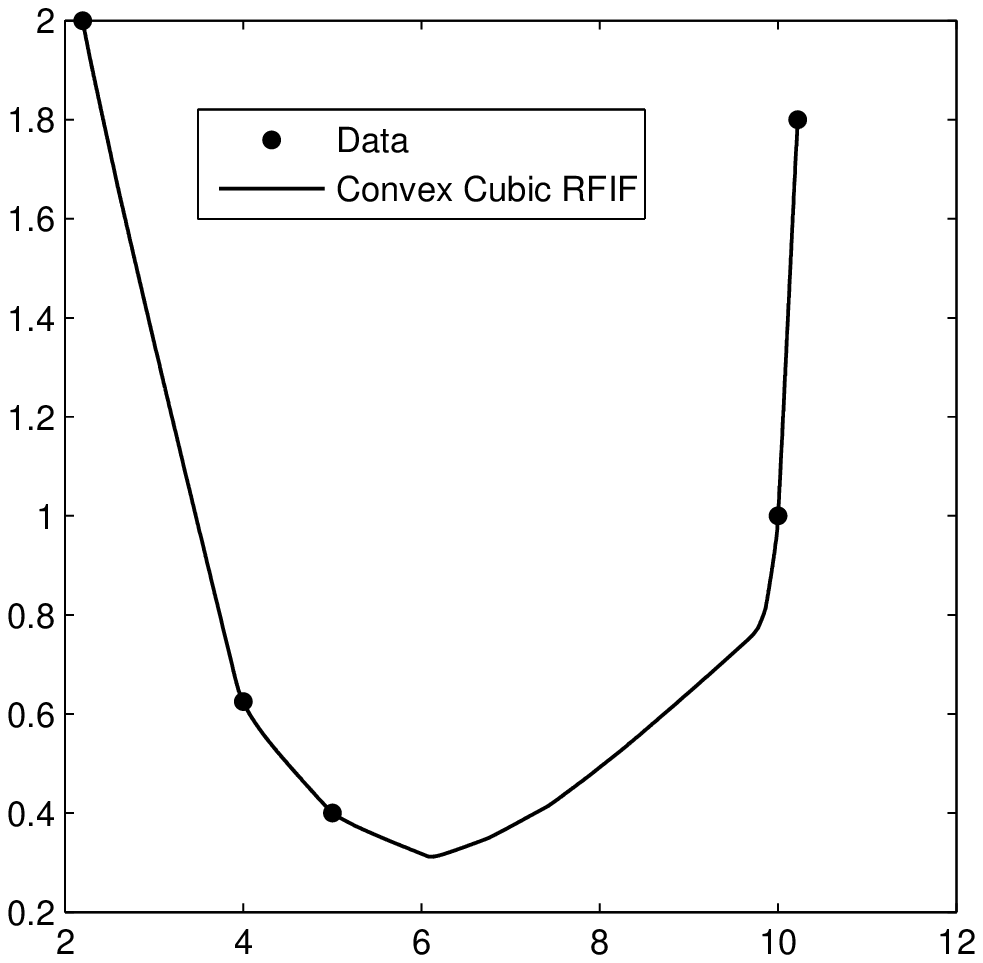,scale=0.38} \centering{\scriptsize{(a):
Convex rational cubic FIF} }
\end{minipage}\hfill
\begin{minipage}{0.3\textwidth}
\epsfig{file=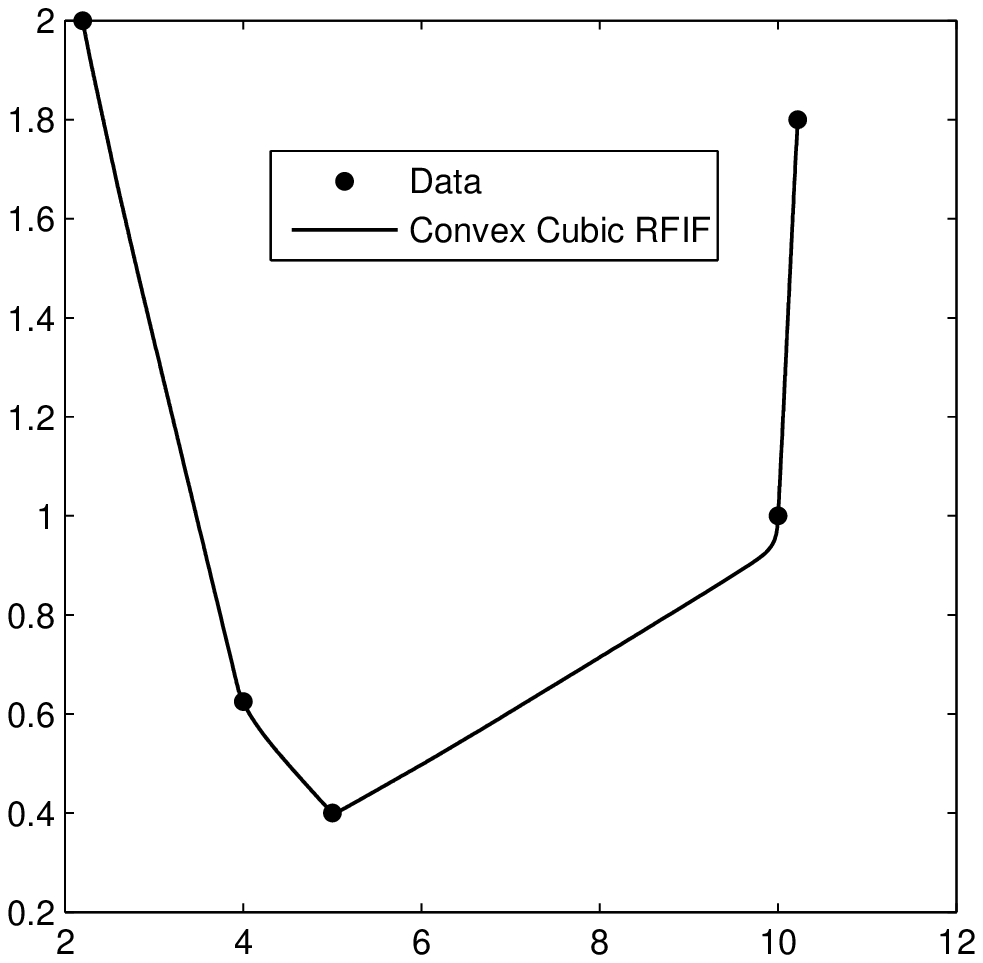,scale=0.38} \centering{\scriptsize{(b):
Convex rational cubic FIF  (effect of $\alpha_3$)}}
\end{minipage}\hfill
\begin{minipage}{0.3\textwidth}
\epsfig{file = 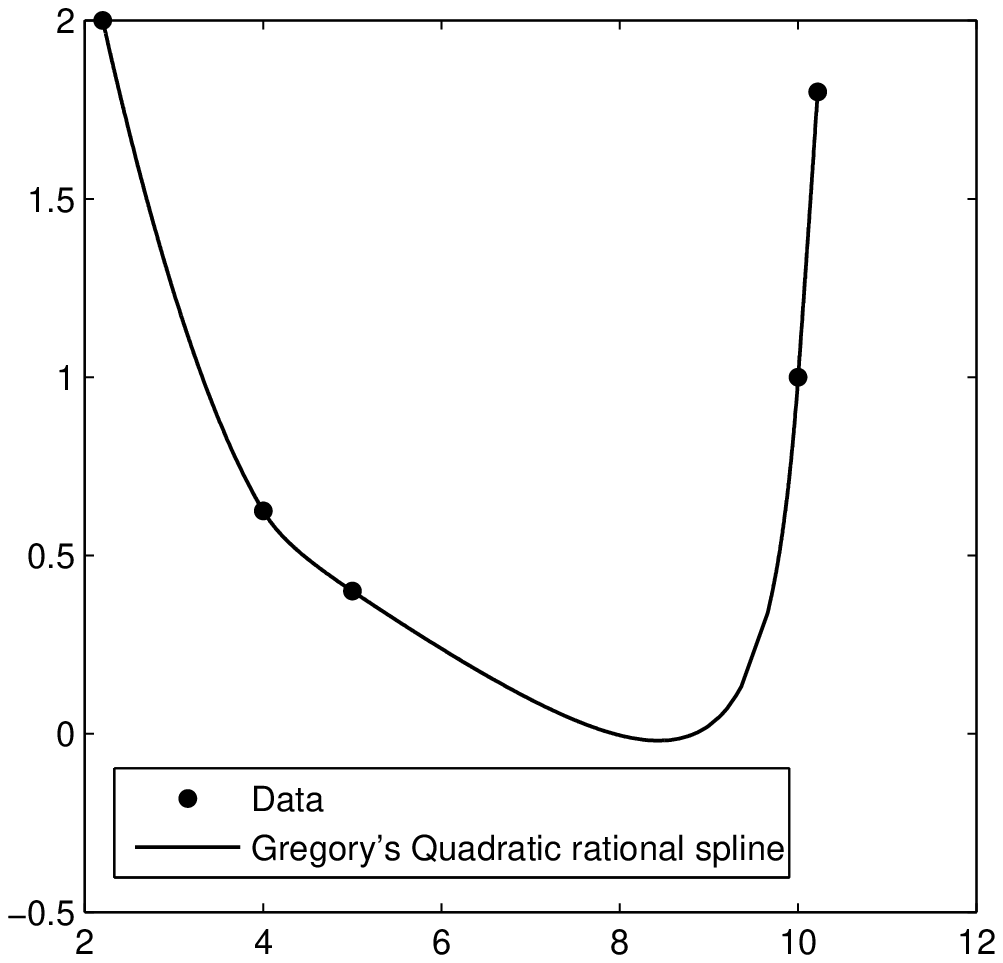,scale=0.38}\\ \centering{\scriptsize{(c):
Classical convex rational quadratic spline}}
\end{minipage}\hfill\\
\caption{Convex rational cubic/quadratic spline FIFs.}
\end{center}
\end{figure}
\subsection{FIFs with Mixed Shape Properties}
 The theoretical discussion we had in section \ref{r1ps6} was confined to
  data with same shape characteristics in the entire interpolation interval.
However, we can also apply  our schemes with proper modification
and mixing to obtain fractal interpolants for data with
mixed shape properties. We illustrate this with two examples.\\
The first example taken from \cite{TM} is a data set generated
from a function $\Phi_1$ defined on $[0,4]$ which is positive on
$[0, 1]$, strictly increasing on $[1, 2]$, constant on $[2, 3]$,
and concave on $[3, 4]$. Suppose that we want to use the proposed
fractal interpolation scheme to construct a $\mathcal{C}^1$
approximant to this function interpolating the data set
$\{(0,1),(0.5,0.2), (1,1), (1.5,1.3), (2,1.8), (2.5,1.8),\\
(3,1.8),(3.5,1.5), (4,0)\}$ with the same shape characteristics.
To achieve this, derivative values that are consistent with the
required shapes are estimated: $d_1=-3.2$, $d_2=d_5=d_6=d_7=0$,
$d_3=1.1$, $d_4=0.8$, $d_8=-1.8$, and $d_9=-4.2$. We divide the
interval into four subintervals of unit length such that in each
of the subinterval data possess the same shape property.  To
obtain a positive rational cubic spline FIF $S_1$ in $[0,1]$, we
iterate the functional equation (\ref{r10}) with the scaling
factors and the shape parameters satisfying the required
conditions \big(see Section \ref{r1ps5} (i)\big). Our specific
choices of the scaling parameters and the shape parameters are:
$\alpha_1=\alpha_2=0.15$; $r_1=1.5$, $r_2=0.5$. On the interval
$[1,2]$, we apply our monotonicity preserving algorithm with the
parameter values $\alpha_1=\alpha_2=0.3$, $r_1=8$, and $r_2=1$ to
obtain the monotonic rational cubic spline FIF $S_2$. Note that
here the parameters are indexed by considering the interpolation
to take place in the subinterval $[1,2]$, not the entire interval.
Following Remark (\ref{rcon}), we generate a linear interpolant
$S_3$ on $[2,3]$. Finally, the concavity preserving rational cubic
spline FIF $S_4$ is obtained on $[3,4]$ by iterating functional
equation (\ref{r10}) with parameters values satisfying concavity
condition, our specific choices being $\alpha_1=0.15$,
$\alpha_2=0.2$, $r_1=9$, and $r_4=4$. Since the data satisfy both
the monotonic decreasing condition and the strictly concave
condition on $[3,4]$, and derivative parameters are selected
accordingly, the concave interpolation scheme automatically render
a concave and monotonically decreasing interpolant (see Section
\ref{r1ps43}). The fractal function $S$  on $[1,4]$ \big(see Fig.
4(a)\big) is obtained by pasting $S_i$, $i=1,2,3,4$. Since $S_i$
and $S_i^{(1)}$ ($i=1,2,3,4$) are continuous, the continuity of
$S$ and $S^{(1)}$ follows from the pasting lemma. Consequently,
the fractal function $S\in
\mathcal{C}^1[0,4]$ given in Fig. 4(a)  provides an approximation to $\Phi_1$ satisfying the required shape properties.\\
Consider a function $\Phi_2$ defined on $[0,10]$, which is
monotonic decreasing on $[0,4]$, monotonic increasing on $[4,6]$,
and monotonic decreasing on $[6,10]$. Our second example is
concerned with the construction of a $\mathcal{C}^1$-approximation
which is co-monotone with this function, where all we know about
the function is the function values at specified points, say,
$\{(0,10),(1.5,5),(4,3.5),\\ (6,7.1), (8,3), (10,0)\}$. As
discussed in section \ref{r1ps5}, we divide the interval in to
three subintervals $I_1=[0,4]$, $I_2=[4,6]$, and $I_3=[6,10]$
where the data are  monotonic decreasing, monotonic increasing,
and monotonic decreasing respectively. Since the interval
$I_2=[4,6]$ contains only two knot points, the FIF scheme demands
insertion of a node in this interval. Let the new node be $(5,6)$.
Derivative values that are consistent with the required shapes are
chosen as $d_1=-4.35$, $d_2=-2.31$, $d_3=0$, $d_4=1.8$ (at the
inserted knot), $d_5=0$, $d_6=-1.77$, and $d_7=-1.2$. A rational
cubic spline FIF $S_1$ is constructed on $I_1$ by taking
$\alpha_1=\alpha_2=0.2$, $r_1=2$, $r_2=12$, and iterating the
functional equation (\ref{r10}). On $I_2$, the functional equation
(\ref{r10}) with the parameter values $\alpha_1=\alpha_2= 0.3$,
$r_1=2$, and $r_2=91$ generates $S_2$. Finally iterations of
(\ref{r10}) with $\alpha_1=\alpha_2=0.4$, $r_1=2$, and $r_2=26$
yield $S_3$. The fractal function $S \in \mathcal{C}^1$ defined in
a piecewise manner by $S\big|_{I_i}=S_i$, $i=1,2,3$ is co-monotone
with $\Phi_2$, and it is given in Fig. 4(b).
\begin{figure}[h!]
\begin{center}
\begin{minipage}{0.5\textwidth}
\epsfig{file=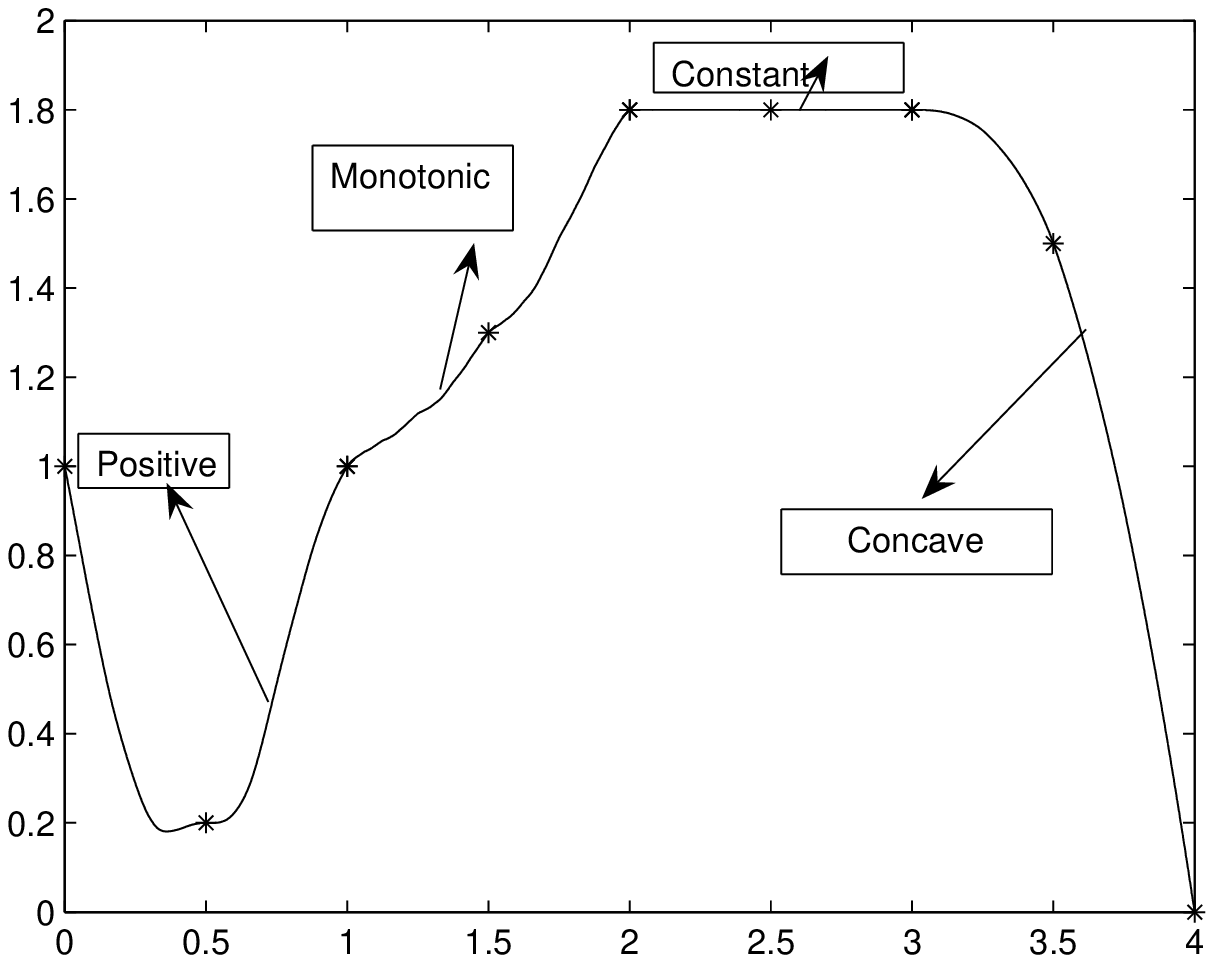,scale=0.5}\\ \centering{\scriptsize{(a):
Rational cubic spline FIF with mixed shape property}}
\end{minipage}\hfill
\begin{minipage}{0.5\textwidth}
\epsfig{file=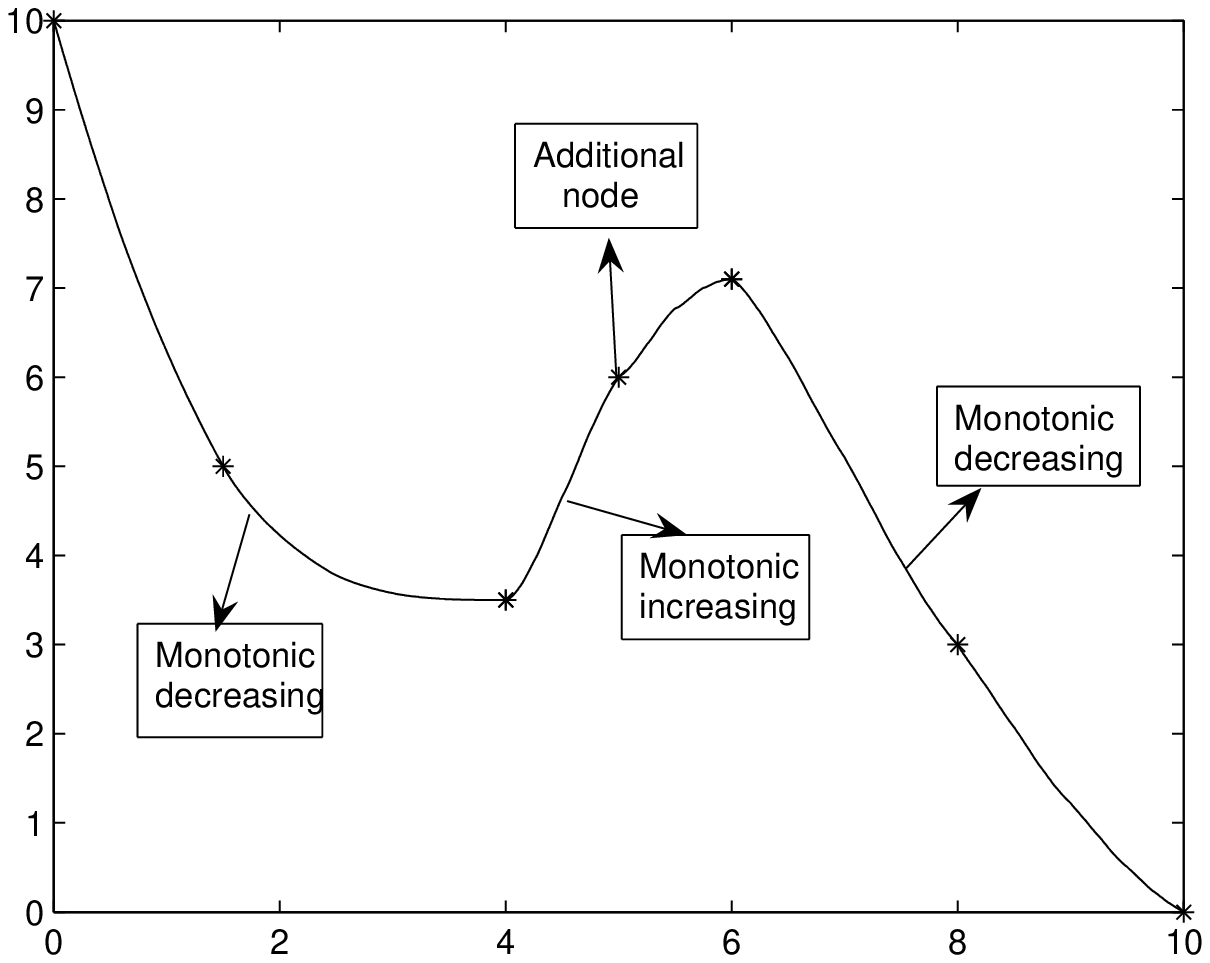,scale=0.5}\\ \centering{\scriptsize{(b):
Co-monotone rational cubic FIF}}
\end{minipage}\hfill
\caption{ Rational cubic spline FIFs with mixed shape properties.}
\end{center}
\end{figure}
\section{Conclusions}
A new kind of rational cubic fractal splines involving shape
parameters is proposed in the present work to provide a tool for
univariate shape preserving interpolation. Number of parameters in
the rational IFS is kept to be minimum (one family of parameters
for controlling fractality in the derivative function and one for
providing shape preserving characteristics) for computational
efficiency. Due to the presence of the scaling factors and the
shape parameters involved in the definition, the proposed
$\mathcal{C}^1$-rational cubic spine FIF  generalizes the
classical rational splines studied in the references
\cite{DG1,DG4,D2}. Despite the implicit and recursive nature of
the FIFs, it is shown that the existence of range restricted
fractal interpolants depends only on the solvability of a finite
set of inequalities resulting from the constraints. These
inequalities are shown to be solvable if the shape parameters are
above and the scaling factors are below certain explicitly
calculable bounds. Uniform convergence of the rational cubic
spline FIF to the original data generating function $f\in$
$\mathcal{C}^4 (I)$ is established. The convergence analysis shows
that $O(h^r) (r =1,2,3,4)$ error bounds can be achieved by
suitable choices of the derivatives, the scaling factors, and the
shape parameters. Thus, the present interpolation method has
convergence properties similar to that of its classical
counterpart, which should be considered along with the flexibility
and diversity offered by the new method. The scaling factors and
the shape parameters can be selected suitably to find an
interpolant satisfying chosen properties such as smoothness,
approximation order, locality, fractality in the derivative, and
shape preservation of the data.  The fairness (visual
pleasantness) of the interpolant can be achieved through a
constrained non-linear optimization. For the shape preserving
interpolants with varying irregularity in the derivatives, the
result is encouraging for the fractal spline class treated in this
paper. Consequently, it is felt that the proposed scheme can
provide an efficient mathematical tool for the simulation of
curves occurring in the study of physical systems, for instance,
in the study of nonlinear control problems such as pendulum-cart
system and  in some fluid dynamics problems such as
 motion of a falling sphere in a
non-Newtonian fluid.\\

 {\bf ACKNOWLEDGEMENTS.} The
first author is thankful to the SERC DST Project No. SR/S4/MS:
694/10 for this work. The second author is partially supported by
the Council of Scientific and Industrial Research India, Grant No.
09/084(0531)/2010-EMR-I.


\begin{thebibliography}{25}

\bibitem{B}   M. F. BARNSLEY, \emph{Fractal functions and interpolations}, Constr. Approx., \textbf{2}(1986),  pp.
303-329.
\bibitem{B1}  M. F. BARNSLEY, \emph{Fractals Everywhere}, Academic Press, Orlando,
Florida,
1988.
\bibitem{BM} M. F. BARNSLEY, J. ELTON, D. HARDIN  AND
P. MASSOPUST, \emph{Hidden variable fractal interpolation
functions}, SIAM J. Math. Ana., \textbf{20(5)}(1989), pp.
1218-1242.
\bibitem{BH}   M. F. BARNSLEY AND A. N. HARRINGTON, \emph{The calculus of fractal interpolation functions}, J. Approx. Theory, \textbf{57}(1989), pp.
14-34.

\bibitem{LB} L. D. BERKOVITZ, \emph{Convexity and Optimization in
$R^{n}$}, John Wiley \& Sons, New York, 2002.


\bibitem{BFK}  W. BOEHM, G. FARIN AND J. KAHMANN,  \emph{A Survey of curve and surface methods in CAGD}, CAGD, \textbf{1}(1984), pp.
1-60.

\bibitem{CK} A. K. B. CHAND AND  G. P. KAPOOR, \emph{Generalized cubic spline fractal interpolation functions}, SIAM J. Numer. Anal.,
\textbf{44} (2006), pp. 655-676.


\bibitem{CK2} A. K. B. CHAND AND G. P. KAPOOR, \emph{Cubic
spline coalescence fractal interpolation through moments},
Fractals \textbf{15(1)}(2007), pp. 41-53.



\bibitem{CK3} A. K. B. CHAND AND G. P. KAPOOR, \emph{Smoothness analysis of coalescence
hidden variable fractal interpolation functions}, Int. J.
Nonlinear Science, \textbf{3(1)}(2007), pp. 15-26



\bibitem{CN1} A. K. B. CHAND AND  M. A. NAVASCU\'{E}S,  \emph{Generalized Hermite fractal interpolation}, Rev. Real Academic de ciencias. Zaragoza, \textbf{64} (2009),  pp.
107-120.



\bibitem{CV} A. K. B. CHAND AND P. VISWANATHAN,  \emph{Cubic Hermite and cubic spline fractal interpolation
functions}, AIP Conference Proceedings, \textbf{1479}(2012), pp.
1467-1470.


\bibitem{SCh} S. CHEN, \emph{The non-differentiability of a class of
fractal interpolation functions}, Acta Math. Sci., 19(4), (1999),
pp. 425-430.



\bibitem{JCC}  J. C. CLEMENT,\emph{ Convexity preserving piecewise
rational cubic interpolation}, SIAM J. Numer. Anal., \textbf{27}
(1990), pp. 1016-1023.


\bibitem{CC}  P. COSTANTINI AND C. MANNI,  \emph{Curve and surface
construction using Hermite subdivision schemes}, J. Comp. Appl.
Math., \textbf{233}(2010), pp. 1660-1673.

\bibitem{DD2} L. DALLA, V. DRAKOPOULOS AND  M. PRODROMOU,  \emph{On the box dimension for a class of
nonaffine fractal interpolation functions}, Anal. Theory Appl.,
\textbf{19}(2003), pp. 220-233.


\bibitem{D2} R. DELBOURGO, \emph{Shape preserving interpolation to convex data by rational functions with quadratic numerator and linear denominator},
IMA J. Numer. Anal., \textbf{9}(1989), pp. 123-126.


\bibitem{DG5}  R. DELBOURGO AND  J. A. GREGORY, \emph{Determination of derivative parameters for a monotonic rational quadratic interpolant}, IMA J. Numer. Anal.,
\textbf{5}(1985), pp. 397-406.

\bibitem{DG2} R. DELBOURGO AND  J. A. GREGORY, \emph{$C^2$ rational quadratic spline interpolation to monotonic data}, IMA J. Numer. Anal., \textbf{3}(1983),  pp.
141-152.
\bibitem{DG4} R. DELBOURGO AND  J. A. GREGORY, \emph{Shape preserving piecewise rational interpolation}, SIAM J. Sci. Stat. Comput., \textbf{6}(1985),
pp. 967-976
\bibitem{DL} N. DYN, D. LEVIN AND  D. LIU,  \emph{Interpolatory convexity preserving
subdivision schemes for curves and surfaces},  Comput. Aided
Design, \textbf{24(4)}(1992), pp. 211-216.



\bibitem{FT} J. C. FIOROT AND J. TABKA,  \emph{Shape-preserving
$C^2$ cubic polynomial interpolating splines}, Math. Comp.,
\textbf{57}(1991), pp. 291-298.

\bibitem{FB}  F. N. FRITSCH AND J. BUTLAND, \emph{A method for constructing local monotone piecewise cubic interpolants}, SIAM J. Sci. Stat. Comput.,
\textbf{5}(1984), pp. 303-304.

\bibitem{FC}  F. N. FRITSCH AND R. E. CARLSON,  \emph{Monotone piecewise cubic interpolation}, SIAM J. Numer. Anal., \textbf{17}(1980),
pp. 238-246.


\bibitem{G} J. A. GREGORY, \emph{Shape preserving spline
interpolation}, CAGD, \textbf{18}(1986), pp. 53-57.

\bibitem{DG1}   J. A. GREGORY AND R. DELBOURGO,  \emph{Piecewise rational quadratic interpolation to monotonic data}, IMA J. Numer. Anal.,
\textbf{2}(1982), pp. 123-130.



\bibitem{GS} J. A. GREGORY AND  M. SARFRAZ, \emph{A rational cubic spline
with tension}, CAGD, \textbf{7}(1990), pp. 1-13.

\bibitem{XH}   X. HAN, \emph{Convexity preserving piecewise rational quartic interpolation}, SIAM J. Numer. Anal., \textbf{46}(2008), pp.
920-929.

\bibitem{JA} A. JAYARAMAN AND A. BELMONTE,  \emph{Oscillations of a solid sphere falling through a wormlike micellar
fluid}, Phys. Rev., \textbf{67}(2003), 65301.

\bibitem{K} W. J. KAMMERER, \emph{Polynomial approximations to finitely oscillating functions},  Math. Comp. \textbf{15}(1961),  pp.
115-119.

\bibitem{LM} L. I. LEVKOVICH-MASLYUK,  \emph{Wavelet-based determination
of generating matrices for fractal interpolation functions},
Regul. Chaotic Dyn., \textbf{3}(1998), pp. 20-29.

\bibitem{LL} E. LUTTON, J. LEVY-VEHEL,  G. CRETIN,   P. GLEVAREC,
AND  C. ROLL, \emph{ Mixed IFS: resolution of the inverse problem
using genetic programming}, Complex systems, \textbf {9}(1995),
pp. 375-398.


\bibitem{TM} T. LYCHE AND J. L.  MERRIEN,  \emph{$C^1$ Interpolatory subdivision with shape
constraints for curves}, Siam J. Numer. Anal., \textbf{44}(2006)
pp. 1095-1121.




\bibitem{MS} C. MANNI AND P. SABLONNI\'{E}RE,  \emph{Monotone
interpolation of order 3 by $C^2$ cubic splines}, IMA J. Numer.
Math., \textbf{71}(1991), pp. 237-252.


\bibitem{N}  M. A. NAVASCU\'{E}S, \emph{ Fractal polynomial interpolation}, Z. Anal. Anwend, \textbf{25}(2005),  pp.
401-418.

\bibitem{NS} M. A. NAVASCU\'{E}S  AND M. V. SEBASTI\'{A}N,  \emph{Generalization of Hermite functions by fractal interpolation},  J. Approx. Theory,
\textbf{131(1)}(2004),
pp. 19-29.

\bibitem{NS2} M. A. NAVASCU\'{E}S AND M. V. SEBASTI\'{A}N,\emph{ A relation between fractal
dimension and Fourier transform -electroencephalographic study
using spectral and fractal parameters}, Int. J. Comp. Math.,
\textbf{85}(2008),  pp. 657-665



\bibitem{S4}  M. SARFRAZ, M. Z. HUSSAIN AND  M. HUSSAIN, \emph{Shape-preserving curve interpolation}, Int. J.  comput. Math.,  \textbf{89}(2012), pp.
35-53.

\bibitem{SCH} D. G. SCHWEIKERT,  \emph{An interpolation curve using a spline in tension}, J. Math. Phys., \textbf{45}(1966),
pp. 312-317


\bibitem{SH} J. W. SCHMIDT  AND  W. HE\'{SS}, \emph{ An always
successful method in univariate convex $C^2$ interpolation},
Numer. Math., \textbf{71}(1995), pp. 237-252.



\bibitem{SC} L. L. SCHUMAKER,  \emph{On shape preserving quadratic spline interpolation}, SIAM J. Numer. Anal., \textbf{20}(1983),
pp. 854-864.



\bibitem{WS} W. M. SIEBERT, \emph{ Circuits, Signals, and Systems}, MIT
press, Cambridge, 1986.



\bibitem{W}  W. WOLIBNER, \emph{Sur un polyn\^{o}me d'interpolation}, Colloq. Math., \textbf{2}(1951),  pp.
136-137.







%








\end{thebibliography}
\end{document}